\documentclass[twoside,12pt]{amsart}
\usepackage[a4paper,scale=.795,centering]{geometry}  
\usepackage{mathrsfs,amsmath,amssymb,mathabx,bm,amsfonts,mathtools}
\usepackage[dvipsnames]{xcolor}
\usepackage[
colorlinks,citecolor=Plum,urlcolor=Periwinkle,linkcolor=DarkOrchid]{hyperref}
\usepackage{pdflscape}
\usepackage{tikz} 
\usetikzlibrary{fit, positioning, arrows.meta, shapes.geometric}

\theoremstyle{plain}
\newtheorem{theorem}{Theorem}

\newtheorem{lemma}[theorem]{Lemma}
\newtheorem{corollary}[theorem]{Corollary}
\newtheorem{proposition}[theorem]{Proposition}

\theoremstyle{definition}

\newtheorem*{remark}{Remark}
\newtheorem{assumption}{\textbf{Assumption}}
\newenvironment{namedassumption}[1]{
  \renewcommand\theassumption{#1}
  \assumption}{\endassumption}
\usepackage{dutchcal} 

\usepackage[foot]{amsaddr}

\title[Local extinctions of branching processes with immigration: scaling limits]{Critical branching processes with immigration: \\
 scaling limits of local extinction sets
}

\author{Aleksandar Mijatovi\'c$^{1,3}$}
\address{1: Department of Statistics, The University of Warwick}
\address{2: Instituto de Matem\'aticas, Universidad Nacional Aut\'onoma de M\'exico}
\address{3: The Alan Turing Institute}
\author{Benjamin Povar$^1$}
\author{Ger\'onimo {Uribe Bravo}$^{1,2}$}

\newcommand{\cond}[2]{\left.\vphantom{#2}#1\ \right| #2}

\DeclareMathOperator{\cbi}{\ensuremath{CBI}}
\DeclareMathOperator{\bgwi}{\ensuremath{BGWI}}

\DeclareMathOperator{\Beta}{B}

\newcommand{\na}{\ensuremath{\mathbb{N}}}
\newcommand{\indi}[1]{\si_{#1}} \newcommand{\si}{{\ensuremath{\bf{1}}}}
\newcommand{\re}{\ensuremath{\mathbb{R}}}
\newcommand{\set}[1]{\ensuremath{\left\{ #1\right\} }}
\newcommand{\proba}[2][{ }]{\ensuremath{\mathbb{P}_{{#1}}\! \left( #2 \right)}}
\newcommand{\paren}[1]{\left(#1\right)}
\newcommand{\bra}[1]{\left[#1\right]}
\newcommand{\esp}[1]{\ensuremath{\mathbb{E}\! \left( #1 \right)}}
\newcommand{\espc}[2]{\ensuremath{\imf{\mathbb{E}}{\cond{#1}{#2}}}}
\DeclarePairedDelimiter\ceil{\lceil}{\rceil}
\DeclarePairedDelimiter\floor{\lfloor}{\rfloor}
\newcommand{\p}{\ensuremath{ \mathbb{P}}}
\newcommand{\cadlag}{c\`adl\`ag}
\newcommand{\eps}{\ensuremath{ \varepsilon}}
\newcommand{\abs}[1]{\hspace{.25mm}\left|#1\right|\hspace{.25mm}}
\newcommand{\imf}[2]{\ensuremath{#1\!\paren{#2}}}

\date{\today}
\begin{document}
\begin{abstract} 
We establish the joint scaling limit of a critical Bienaym\'e-Galton-Watson process with immigration (BGWI) 
and its (counting) local time at zero
to the corresponding self-similar continuous-state branching process with immigration (CBI) 
and its (Markovian) local time at zero
for balanced offspring and immigration laws  in stable domains of attraction.
Using a general framework for invariance principles of local times~\cite{MR4463082}, the problem reduces to the  
analysis of the structure of excursions from zero and positive levels, together with
the weak convergence of the hitting times of points of the BGWI to those of the CBI. 
A key step in the proof of our main limit theorem is a novel Yaglom limit for the law at time $t$ of an excursion  with lifetime exceeding $t$ of a scaled  infinite-variance critical BGWI. 

Our main result implies a joint septuple scaling limit of BGWI $Z_1$, its local time at $0$, the random walks $X_1$ and $Y_1$ associated to the reproduction and immigration mechanisms, respectively, the counting local time at $0$ of $X_1$, an additive functional of $Z_1$ and $X_1$  evaluated at this functional. In the septuple  limit, four different scaling sequences are identified and given explicitly in terms of the offspring generating function (modulo asymptotic inversion), the local extinction probabilities of the BGWI and the tails of return times to zero of $X_1$.
\end{abstract}
\maketitle

\section{Introduction and main results}
\label{sectionIntroduction}
\subsection{Scaling limits of the critical BGWIs and their counting local times at zero}
It is 
known that the scaling limits of Bienaymé-Galton-Watson processes with immigration ($\bgwi$) 
are self-similar continuous-state branching processes with immigration ($\cbi$), 
see \cite{MR0290475,MR2225068}, \cite{MR3098685} and Theorem~\ref{septupleLimitTheorem} below. 
However, little is known about the convergence of natural functionals of the $\bgwi$ processes,
such as their local times 
(as representative aspects of the structure of their zero sets), 
to their continuum analogues. 
The main aim of this paper is to provide results in this direction. 
The setup naturally includes 
branching processes with infinite immigration mean and
critical reproduction mechanism with infinite  offspring variance. 
Some of our results are well known and classical in the finite-variance setting without immigration  
 (known by names such as Feller's diffusion, Kolmogorov's estimate or Yaglom's limit).  
In particular,
 the convergence of one-dimensional distributions has been investigated in \cite{MR0290475} and,   conditioned on non-extinction in the finite variance case,  in \cite{MR0266320} 
and \cite{MR451433}. 
For the analysis of the asymptotic tails of the life spans of BGWI processes see~\cite{MR0300351},  \cite{MR0688076}, \cite{MR451433},\cite{MR0776902}. 
The cases of infinite offspring variance  or infinite immigration mean  appear not to have been analysed as much. 
More recent results, still in the finite-variance reproduction case, 
have considered random environment or state-dependent BGW processes, as in \cite{MR3834847,MR4094392,MR4259452}. 
On the other hand, limit theorems for local times have been studied for random walks 
or reflected random walks 
(as in \cite{MR665738}, \cite{MR636771}, \cite{MR749918}, \cite{MR2663630}, \cite{MR3485370}). 
More recently, 
a general framework to obtain invariance principles for local times was put forth in \cite{MR4463082}; 
it is this general framework that we will apply in the setting of branching processes with immigration.

Let $Z_1$ denote a $\bgwi$ process started at $0$ with the offspring and immigration distributions $\mu$ and $\nu$ on $\na\coloneqq\{0,1,\ldots\}$, respectively (see~\eqref{eq_discreteLampertiTransformation_1} below for a definition of BGWI).
We assume throughout the paper 
$\mu(1)<1$
and 
$\nu(0)>0$,
implying that the reproduction mechanism is not deterministic and that $Z_1$ 
may return zero. 
Following \cite{MR0096321,MR0228077, MR0290475}, we now introduce Assumption~\ref{assumption_SL}.

\begin{namedassumption}{\textbf{(SL)}}\label{assumption_SL}
The generating functions
\[
    f(s)\coloneqq \sum_{n\in\na}s^n \mu(n)\quad\text{and}\quad g(s)\coloneqq \sum_{n\in\na}s^n \nu(n),\qquad s\in[0,1],
\]
of the offspring and immigration distributions $\mu$ and $\nu$ 
take the form
\begin{equation}
    \label{eq:critical_gen_functions}
	f(s)=s+c(1-s)^{1+\alpha}\mathcal{l}(1-s)
	\quad\text{and} \quad
	g(s)=1-d(1-s)^\alpha\mathcal{k}(1-s),\quad s\in (0,1), 
\end{equation}    
for some constants $c,d>0$ and $\alpha\in (0,1]$ 
and functions $\mathcal{l},\mathcal{k}:(0,1)\to(0,\infty)$, slowly varying at zero, with $\mathcal{l}(s)/\mathcal{k}(s)\to 1$ as $s\to 0$.
\end{namedassumption}

As shown in \cite[Thm~2.3]{MR0290475}, \cite{MR2225068}, \cite[Thm~5.8]{MR4339413} or \cite{MR3098685} 
(see also the proof of Theorem~\ref{septupleLimitTheorem} below), 
Assumption~\ref{assumption_SL} is equivalent to the existence of a scaling limit for $Z_1$, which we now describe.
Define a scaling sequence $(b_n)$
via asymptotic inversion  of the regularly varying function $x\mapsto x^\alpha/\mathcal{l}(1/x)$ at infinity~\cite[Thm~1.5.12]{MR898871}. Thus $(b_n)$ is regularly varying with index $1/\alpha$
and satisfies 
$b_n^\alpha /(n\mathcal{l}(1/b_n))\to1$  as $n\to\infty$, denoted  throughout the paper  by
\begin{equation}
    \label{equationDefbn}
    b_n^\alpha \sim n\mathcal{l}(1/b_n)\quad\text{as $n\to\infty$.}
\end{equation}
For  $u\in \re$, let
$\lfloor u\rfloor\coloneqq \max\{n\in\mathbb{Z}:n\leq u\}$.
The continuous-time extension of the scaled $\bgwi$, 
\begin{equation}
\label{eq:cont_time_bgwi}
        \frac{1}{b_n} Z_1(\floor{n\cdot}),
\end{equation}
converges weakly as $n\to\infty$ in Skorokhod space (of \cadlag\  functions mapping $[0,\infty)$ into itself)   to a process $Z$, which is a self-similar continuous-state branching process with immigration~(CBI) of index $\alpha$
started at $0$ (see~\eqref{eq_LampertiTransformation} below for a definition of a CBI process).\footnote{See \cite{MR4339413} for an introduction to $\cbi$ processes. 
The weak convergence $Z_1(\floor{n\cdot})/b_n\stackrel{d}{\to} Z$ follows from the finite dimensional convergence \cite[Thm~2.3]{MR0290475} and the tightness of \cite{MR2225068}. 
\cite{MR0290475}  also establish the necessity of Assumption~\ref{assumption_SL} for the finite dimensional distributions of $Z_1(\floor{n\cdot})/b_n$ to converge weakly. The limit $Z_1(\floor{n\cdot})/b_n\stackrel{d}{\to} Z$ also follows under Assumption~\ref{assumption_SL} as part of the time-change coupling construction in Subsection~\ref{subsec:septuple_lim_thm} below, given before the proof of our septuple limit result (see  Theorem~\ref{septupleLimitTheorem}).}
The law of the self-similar CBI $Z$ is
characterised by the Laplace transform of its one-dimensional distributions with arbitrary initial state $z\in[0,\infty)$:
\begin{equation}
\label{equationLaplaceTransformOfCBI}
	\imf{\mathbb{E}_z}{e^{-\lambda Z(t)}}= (1+\alpha c \lambda^\alpha t)^{-\delta} 
        e^{-\frac{\lambda z}{(1+\alpha c \lambda^\alpha t)^{1/\alpha}}},\quad\text{ where }
        \delta\coloneqq\frac{d}{\alpha c}. 
\end{equation}
The finite-variance case $\alpha=1$ is special in that $Z$ has continuous sample paths; 
indeed, when $c=2$ it is a squared Bessel process of dimension $2\delta$ and a multiple of it otherwise.

Our main result gives a scaling limit of the zero (or local extinction) set $\set{m \in\na: Z_1(m) = 0}$  of $Z_1$, 
encoded via the counting 
local time $L_1$ of $Z_1$ at $0$, 
 defined for any $t\geq 0$ as
\begin{equation}
\label{equationDefiningDiscreteLocalTime}
L_1(t)  \coloneqq  \left|\set{m \in\na: Z_1(m) = 0}\cap [0, t]\right|.
\end{equation}
The main aim is to prove that the scaling limit of $L_1$ is the Markovian local time 
$L$ at $0$ of the self-similar CBI $Z$ started at $0$.
This local time ${L}$ 
is non-trivial if and only if $\delta \in (0,1)$ (see \cite[\S 5.2.1]{MR3263091}). 
Indeed, recall that $Z$ is point recurrent at $0$ if and only if $\delta \in (0,1)$.  
If $\delta\geq 1$, 
$Z$ does not return to $0$ at positive times. 
In the particular case of a squared Bessel process, i.e. $\alpha=1$ and $c=2$, this dichotomy is manifested through its dimension $2\delta$.

When $\delta\in (0,1)$, the right continuous inverse of ${L}$ is a stable subordinator of index $1-\delta$. 
Our main result 
adds a limit theorem to 
the celebrated determination of the structure of the zero set of a squared Bessel process~\cite{MR0247668}, 
corresponding to the special case $\alpha=1$ and $c=2$ of the preceding discussion. 

\begin{theorem}
\label{theorem_Local_time_conv}
Let Assumption~\ref{assumption_SL} hold with 
$\delta=\frac{d}{\alpha c}\in (0,1)$.
Then, for any 
$\kappa>0$, 
the sequence $(c_n)$ given by 
\begin{equation}
\label{eq_theorem_Local_time_conv_c_n}
c_n \coloneqq \kappa n \proba{Z_1(n)=0},  \quad n\in\na,
\end{equation}
is regularly varying of index $1-\delta$. Furthermore,  there exists $\kappa$ so that the weak convergence
\begin{equation}
\label{eq_theorem_Local_time_conv_joint_conv}
    \paren{
    \frac{1}{b_n}Z_1(\floor{n\cdot}),\frac1{c_n}{L_1(\floor{n\cdot})}
    } 
    \stackrel{d}{\to } (Z, {L}) \ \text{as} \ n\to\infty
\end{equation}holds in the product Skorokhod topology. 
\end{theorem}
\begin{remark}
\noindent (a) Both the scaling sequence $\paren{c_n}$ in Theorem~\ref{theorem_Local_time_conv} 
and the Markov local time $L$ of $Z$
are defined up to positive multiplicative constants only. 
Thus, fixing the scaling sequence determines the multiplicative constant in the definition of the local time $L$ and vice versa. 

\noindent (b) As shown in Proposition~\ref{porp:A1} of Appendix~\ref{subsec:equivalence_SL_SL'} below,  in the case $\alpha\in (0,1)$, Assumption~\ref{assumption_SL} is equivalent to the following: 
the offspring distribution
$\mu$ is \emph{critical} (i.e.~has mean one), 
its tail $\overline \mu(k)\coloneqq \mu((k,\infty)\cap\na)$ is regularly varying with index $-(1+\alpha)$ and  the following tail balance condition holds:
\[
\frac{\overline \nu(k)}{k \overline \mu(k)}\text{ converges to a limit in $(0,\infty)$ as $k\to\infty$.}
\]
The tail $\overline \nu(k)\coloneqq \nu((k,\infty))$ is also regularly varying  of index $-\alpha$, placing $\mu$ and $\nu$ in the $(1+\alpha)$-stable and $\alpha$-stable domains of attraction, respectively, and making Theorem~\ref{theorem_Local_time_conv} applicable to the entire class.

\noindent (c) Any critical finite-variance offspring generating function satisfies Assumption~\ref{assumption_SL}
with $\alpha=1$, in which case 
the slowly varying function $\mathcal{l}$  is bounded on a neighbourhood of zero. Moreover, in this case, the corresponding immigration has finite mean and $\mathcal{k}$ is likewise locally bounded at zero (see Lemma~\ref{lemma_infinite_activity} 
in Appendix~\ref{app:alpha=1} below). 
Thus Theorem~\ref{theorem_Local_time_conv} covers all critical BGWIs with finite-variance offspring distributions. 
Note further that Assumption~\ref{assumption_SL} with $\alpha=1$ does not require the offspring distribution to have finite variance (see the Remark 
in Appendix~\ref{app:alpha=1} below). This case is also covered by Theorem~\ref{theorem_Local_time_conv}.

\noindent (d) The convergence of the counting local time in~\eqref{eq_theorem_Local_time_conv_joint_conv} is fragile. The  example in 
Appendix~\ref{app:example_non_conv_local_time} below gives a convergent sequence of BGWIs, with the same reproduction mechanism as  $(Z_1(\floor{n\cdot})/b_n)_{n\geq1}$  and only a local perturbation of the immigration law $\nu$ at $0$, such that its limit is the self-similar CBI process $Z$ (as in  Theorem~\ref{theorem_Local_time_conv}) but  the corresponding counting local times are proved not to converge to the Markov local time of $Z$ at zero.
\end{remark}

We will obtain the joint convergence of the local time process, together with the scaled BGWI process, using the recently established general framework of \cite{MR4463082}.
We now state a simplified version of~\cite[Thm~1]{MR4463082} in the context of the models studied in the present paper, which
satisfy Assumption~\ref{assumption_SL}  with $\delta=\frac{d}{\alpha c}\in(0,1)$.
Note that the convergence of local time in Theorem~\ref{thmLocalTimeLimitTheoremFromMUB22} below is  in probability. 

\begin{theorem}[{\cite[Thm 1]{MR4463082}}]
\label{thmLocalTimeLimitTheoremFromMUB22}
\label{theorem_MUB22}
Let Assumption \ref{assumption_SL} hold 
with $\delta=\frac{d}{\alpha c}\in(0,1)$ and the  sequence $(b_n)$ satisfy~\eqref{equationDefbn}.
Let $(Z_n)$ be a sequence of continuous-time stochastic processes on the same probability space as  $Z$, such that $Z_n\stackrel{d}{=}Z_1(\floor{n\cdot})/b_n$. 
For any $t\in[0,\infty)$, define
\begin{equation}   \label{eq_thmLocalTimeLimitTheoremFromMUB22_g_t_d_t}
    g_t(Z)\coloneqq \sup\set{s\leq t: Z(s)=0}
    \quad\text{and}\quad
    d_t(Z)\coloneqq \inf\set{s> t: Z(s)=0},
\end{equation}
and the corresponding $g_t(Z_n)$
and $d_t(Z_n)$. 
Define $L_n(t)
    :=\left|\set{s\in \na/n: Z_n(s)=0}\cap[0,t]\right|$.
If  $Z_n\stackrel{\p}{\to} Z$ and, for every $t>0$, 
\begin{equation}\label{eq_theorem_MUB22_a_assumption}
g_t(Z_n)\stackrel{\p}{\to}g_t(Z) \quad\text{ and } \quad d_t(Z_n)\stackrel{\p}{\to} d_t(Z),\quad\text{as $n\to\infty$,}
\end{equation}%
there exists a constant $\tilde \kappa>0$ such that for the scaling sequence 
   $ \tilde c_n
    \coloneqq\tilde \kappa/\proba{d_{1/n}(Z_n)>1}$ the following limit in probability holds:
\[(Z_n,L_n/\tilde c_n)\stackrel{\p}{\to} (Z,L) \text{ as }n\to\infty .\] 
\end{theorem}

We will see in the proof of Theorem \ref{theorem_Local_time_conv} 
(cf. equations \eqref{equationLocalLimitTheoremBGWIAt0} and \eqref{equationExcursionLengthAsymptoticsBGWI} of Lemma \ref{lemmaPreLimitAsymptotics} below) 
that the sequences $(c_n)$ and $(\tilde c_n)$ are asymptotically equivalent. 
Hence, to apply Theorem \ref{theorem_MUB22}, we need to verify the limits in~\eqref{eq_theorem_MUB22_a_assumption} of the endpoints of excursions of the scaled BGWI process. 
The applications in~\cite{MR4463082} of the general form of Theorem~\ref{thmLocalTimeLimitTheoremFromMUB22} for regenerative processes are mainly path-wise and do not apply in our setting for two reasons: 
the $\bgwi$ process $Z_1$ is not downwards skip-free and the CBI process $Z$ is  not downwards regular at zero (since it is non-negative).
This makes the hitting times by $Z$ of the boundary point zero much harder to detect from the hitting times of zero by $Z_1$. 

One of the main contributions in this paper is to present a new paradigm, combining path-wise and distributional arguments, enlarging the scope of applications of the framework in~\cite{MR4463082}. 
In the setting of the present paper, this new paradigm rests on the branching properties of both  the $\bgwi$  $Z_1$ and the CBI $Z$, as well as the self-similarity of  $Z$. In particular, we establish the following Yaglom limit for infinite-variance critical branching processes with immigration. 

\begin{theorem}[Yaglom limit for BGWI]
\label{lemmaYaglomLimit}
Under Assumption \ref{assumption_SL} with $\delta=\frac{d}{\alpha c}\in(0,1)$, 
fix any $t>0$ and consider a sequence $(t_n)$ such that $t_n\to t$. 
Then, $Z_n(t_n)$, conditioned on $Z_n$ remaining positive on $[1/n,t_n]$, 
converges weakly to the Linnik law $\nu^t$:
\begin{equation}
    \label{equationYaglomLimit}
    \lim_{n\to\infty}
    \espc{e^{-\lambda Z_n(t_n)}}{d_{1/n}(Z_n)>t_n}
    =(1+\alpha c t \lambda^\alpha)^{-1}\quad \text{for all $\lambda\geq0$.} 
\end{equation}
\end{theorem}

While Assumption \ref{assumption_SL} covers all finite variance offspring distributions (see Appendix~\ref{app:alpha=1} below), the novelty in Theorem~\ref{lemmaYaglomLimit} is in the infinite-variance case (see the classical finite-variance result in~\cite[Thm~1]{MR451433}). 
Theorem~\ref{lemmaYaglomLimit} plays a key role in the proof of Theorem~\ref{theorem_Local_time_conv} and is of independent interest.
The limiting Laplace transform in \eqref{equationYaglomLimit} is that of a Linnik law $\nu^t$ (cf.~\cite{MR1047827}).
Note that the Laplace transform in~\eqref{equationYaglomLimit} coincides with that of the marginal at time $t$ (with $\delta=1$) under $\p_0$ in \eqref{equationLaplaceTransformOfCBI} but, unlike its unconditioned version,
surprisingly  depends neither on $d$ nor on $\delta$.
The presence of the immigration in the BGWI $Z_n$ persists in the limit in~\eqref{equationYaglomLimit}: the corresponding Yaglom limit for BGW process in~\cite[Thm~1]{MR0228077} is not Linnik (for $\alpha\in(0,1)$).
Note also that the Linnik law $\nu^t$ is a natural generalisation of the corresponding Yaglom limit for finite-variance critical BGWIs ($\alpha=1$ in our setting), where the limit is exponential~\cite[Thm~1(a)]{MR451433}.

It is natural to consider an extension of Theorem~\ref{lemmaYaglomLimit} to the limit of the law of the entire excursion of $Z_n$, straddling the time $1/n$, with lifetime greater than $t_n$. However, as this is unnecessary for the scaling limits of local extinctions of BGWIs, it is left for future research. 

Two further results, crucial in our proof of Theorem~\ref{theorem_Local_time_conv}, are the local limit theorem for $g_t(Z_n)/t$ with generalised arcsine limit law of parameter $1-\delta$ (see Corollary~\ref{corPreLimitAsymptotics} below for the precise statement) and the characterisation of the conditional law of the marginal of the CBI process $Z(t)$, given $g_t(Z)=s$, 
as the Linnik law $\nu^{t-s}$. The local limit theorem follows from  the asymptotics of both the local extinction probabilities of the BGWI and the tails of its return times to zero. The conditional law has a short direct proof using excursion theory and the self-similarity of the CBI process $Z$ (see Proposition~\ref{lemmaCBIExcursionStraddlingt} below for details). Since this conditional law equals the limit law in~\eqref{equationYaglomLimit}, Proposition~\ref{lemmaCBIExcursionStraddlingt} and Theorem~\ref{lemmaYaglomLimit}
imply that excursions with lifetime greater than $t$ (evaluated at time $t$) of BGWIs converge to those of the limiting CBI.

A short \href{https://youtu.be/M6cBiiJt_90?si=VRRoRmTYZMzTjat4}{YouTube presentation}~\cite{YouTube_talk} describes our results. \href{https://youtu.be/PCuIkPBApoE?si=abwrWDXAy9uYOqO4}{Part 2} of this video discusses the ideas behind the proofs as well as their structure.

\subsection{The septuple limit theorem}
\label{subsec:septuple_lim_thm}
Before proceeding to the proof of Theorem~\ref{theorem_Local_time_conv} in Sections~\ref{sectionProofOverview} and~\ref{sectionProofsOfAuxiliaryResults} below,   we explain
 why $Z_1(\floor{n\cdot})/b_n \stackrel{d}{\to} Z$ as $n\to\infty$ and construct a probability space on which the convergence takes place in probability (and not only weakly) as Theorem~\ref{theorem_MUB22} requires. 
The latter can of course be achieved through Skorokhod's representation for weak convergence, 
but a much more concrete construction (in terms of the scaling limits of random walks) yields a far-reaching extension of Theorem~\ref{theorem_Local_time_conv},  given in Theorem~\ref{septupleLimitTheorem} below, hard to obtain through other methods. 

Define the distribution $\tilde \mu$ on $\na\cup\{-1\}$ by $\tilde \mu(k)=\mu(k+1)$ and let 
$X_1$ and $Y_1$ be independent random walks  with jump distributions $\tilde \mu$ and $\nu$, respectively. 
 The discrete Lamperti transformation 
 (see, for example, \cite[Eq.~(1.1) in Ch. 9\S 1]{MR838085}, \cite[p. 2]{MR3444314} or \cite[Eq.~(1)]{MR3098685})
 tells us that $Z_1$ has the same law as the solution to the recursion
\begin{equation}
\label{eq_discreteLampertiTransformation_1}
    Z_1= X_1\circ C_1+Y_1,
    \quad\text{where}
    \quad
    C_1(k)=\sum_{0\leq j<k} Z_1(j)\qquad\text{for all $k\in\na$.}
\end{equation}
As remarked above, under Assumption~\ref{assumption_SL}, $\mu$ is critical, implying that $X_1$ is a centred random walk. 
Recall that $X_1$ and $Y_1$ are in the domain of attraction of (independent) stable L\'evy processes $X$ and $Y$ with positive jumps 
of indices $1+\alpha$ and $\alpha\in (0,1]$, respectively (see also Appendix \ref{subsec:equivalence_SL_SL'}). 
More specifically, for the scaling sequence $\paren{b_n}$ in~\eqref{eq:cont_time_bgwi} above, we have
\[
        \left(X_1(\floor{n b_n \cdot}), Y_1(\floor{n\cdot})\right)/b_n\stackrel{d}{\to}  (X,Y)\quad\text{as $n\to\infty$}
\]
in the product Skorokhod topology. 
Since $X$ is a spectrally positive stable L\'evy process and $Y$ is a stable subordinator, 
 their respective Laplace exponents take the form $\Psi(\lambda)=c\lambda^{1+\alpha}$ and $\Phi(\lambda)=d\lambda^\alpha$, $\lambda\geq 0$. 
Note that, when $\alpha=1$, $X$ is a multiple of Brownian motion and $Y$ is the deterministic subordinator $t\mapsto d\cdot t$. 
Heuristically (and rigorously established in Theorem~\ref{septupleLimitTheorem} below), 
the discrete Lamperti transformation of $X_1$ and $Y_1$ in~\eqref{eq_discreteLampertiTransformation_1}
can be scaled to converge 
to the continuous Lamperti transformation $Z$ of $X$ and $Y$, 
the unique solution to the equation
\begin{equation}
\label{eq_LampertiTransformation}
	Z=X\circ C+Y,
        \quad\text{where}\quad 
        C(t)=\int_{0}^t Z(s)\, ds\qquad\text{for all $t\geq0$.}
\end{equation}
The process $Z$ defined in~\eqref{eq_LampertiTransformation} is a continuous-state branching process 
whose dynamics are characterised by their Laplace transforms 
in~\eqref{equationLaplaceTransformOfCBI} above. 
By analogy with the discrete case, the Laplace exponents of $X$ and $Y$ are then called the branching and immigration mechanisms of $Z$.\footnote{
See \cite{MR3098685, MR3689968} for an analysis of the time-change equation.}
The above method of establishing weak convergence of branching processes in terms of associated random walks is a path-wise version of results of \cite{MR0362529} when there is no immigration. 
We can now state the following septuple limit theorem.

\begin{theorem}
\label{septupleLimitTheorem}
Let Assumption \ref{assumption_SL} hold 
with $\delta=\frac{d}{\alpha c}\in(0,1)$ and the regularly varying scaling sequence $(b_n)$ satisfy~\eqref{equationDefbn}.
Consider  independent random walks $X_1$ and $Y_1$ whose jump distributions have generating functions $f(s)/s$ and $g(s)$ and independent stable L\'evy processes $X$ and $Y$ with Laplace exponents $\Psi(\lambda)=c\lambda^{1+\alpha}$ and $\Phi(\lambda)=d\lambda^\alpha$. 
Let $Z_1$ and $C_1$ (resp. $Z$ and $C$) be constructed from $X_1$ and $Y_1$  (resp. $X$ and $Y$) as in 
\eqref{eq_discreteLampertiTransformation_1} (resp.~\eqref{eq_LampertiTransformation}). 
Let $L_1(X_1)$ 
and $L_1(Z_1)$ 
be the counting local times  at $0$
of $X_1$
and $Z_1$ (cf.~\eqref{equationDefiningDiscreteLocalTime}) and let $L(X)$ and $L(Z)$ be (Markovian) local times at $0$ of $X$ and $Z$. 
Then there exists a regularly varying sequence $\paren{a_n}$ of index 
$1+1/\alpha$,  
such that 
\begin{align*}
    &\paren{ \frac{X_1(\floor{nb_n\cdot})}{b_n}, 
    \frac{L_1(X_1)(\floor{n b_n \cdot})}{a_{n b_n}}, 
    \frac{Y_1(\floor{n\cdot})}{b_n},
    \frac{C_1(\floor{n\cdot})}{n b_n}, 
    \frac{X_1\circ C_1(\floor{n \cdot})}{b_n}, 
    \frac{Z_1(\floor{n\cdot })}{b_n}, 
    \frac{L_1(Z_1)(\floor{n\cdot})}{c_n}
    }
    \\&\stackrel{d}{\to} 
    (X,L(X), Y,C,X\circ C,Z,L(Z))\quad\text{as $n\to\infty$.}
\end{align*}
\end{theorem}
Note the scaling in the second component $L_1(X_1)$:
in general the temporal scaling would be $n$ and the spatial scaling would be $a_n$. In this case we pass to a subsequence for the joint convergence of the random vector. 
Also, note that the regularly varying sequence $nb_n$ can always be assumed to be integer valued 
(since $\floor{b_n}$ also satisfies \eqref{equationDefbn} if $b_n$ does). 
The simplest scalings are those of $Y_1$ and $Z_1$, which then dictate the form of the others  as seen in the proof. 

The above result is essentially a straightforward consequence of Theorem~\ref{theorem_Local_time_conv}.  
The hard part of the proof of Theorem~\ref{theorem_Local_time_conv} consists of establishing the assumptions of Theorem~\ref{thmLocalTimeLimitTheoremFromMUB22}. 
Assuming we have done so, we now proceed with a simple proof of Theorem~\ref{septupleLimitTheorem}. This includes a proof of the weak convergence of the scaled BGWI processes to the CBI processes, 
which is essentially (but not explicitly) given in ~\cite{MR3098685}. 
This convergence depends on the following explicit coupling. 


The Laplace transforms
$\lambda\mapsto g(e^{-\lambda/b_n})^n$ and $\lambda\mapsto e^{\lambda \lfloor nb_n\rfloor/b_n}f(e^{-\lambda/ b_n})^{\lfloor nb_n\rfloor}$ (of $Y_1(n)/b_n$ and $X_{1}(\floor{n b_n})/b_n$) converge towards $\exp(-d\lambda^\alpha)$ and $\exp (-c\lambda^{1+\alpha})$  as $n\to\infty$, respectively (see~\eqref{eq:g_conv} and~\eqref{eq:f_conv} in Appendix \ref{subsec:equivalence_SL_SL'}).
Thus we see that these random variables converge to $X(1)$ and $Y(1)$ as $n\to\infty$. 
Skorokhod's theorem~\cite[Thm~23.14]{MR4226142}
then implies the weak convergence of the scaled processes $X_1(\floor{nb_n\cdot})/b_n\stackrel{d}{\to } X$ and $Y_1(\floor{n\cdot })/b_n\stackrel{d}{\to } Y$ in the Skorokhod topology. 
Independence implies their joint weak convergence to  $(X,Y)$ in the product Skorokhod topology. 
From this point onwards, we work on a probability space where this convergence holds almost surely: 
assume then, that on an adequate probability space, there exist a sequence of processes $X_n$ and $Y_n$ such that $X_n\stackrel{d}{=}X_1(\floor{nb_n\cdot})/b_n$ and $Y_n\stackrel{d}{=} Y_1(\floor{n\cdot })/b_n$ and that $X_n\to X$ and $Y_n\to Y$ almost surely (in the Skorokhod topology). 
Note that the  temporal scaling of the process $X_1$, with law equal to that of $X_n$, 
allows us to apply~\cite[Lem.~6]{MR3098685} and obtain the scalings that should be used for $Z_1$ and $C_1$. 
Indeed, in the terminology of~\cite{MR3098685}, 
for a recursion $h^{1/n}$, analogous to the discrete Lamperti transformation in~\eqref{eq_discreteLampertiTransformation_1},
but in which time increases each $1/n$ (i.e. $1/n$ is a discretisation parameter), we have 
\[
    \frac{1}{b_n}Z_1(\floor{n\cdot})\stackrel{d}{=}Z_n=h^{1/n}(X_n,Y_n)\qquad\&\qquad
     \frac{1}{nb_n}C_1(\floor{b_n\cdot})\stackrel{d}{=} C_n =c^{1/n}(X_n,Y_n)
\]
(see~\cite[Lem.~6]{MR3098685} for the latter recursion).
Assuming~\eqref{eq_LampertiTransformation}, together with a certain differential inequality, has a unique solution and  the processes $X\circ C$ and $Y$ do not jump at the same time,~\cite[Thm~3]{MR3098685} gives the convergence $Z_n\to Z$ and $C_n\to C$, 
almost surely. 
The fact that \eqref{eq_LampertiTransformation} and the differential inequality have a unique solution 
follows from~\cite[Prop.~2]{MR3098685}. 
By \cite[Prop.~4]{MR3098685}, $X\circ C$ and $Y$ do not jump at the same time since $X$ is quasi left continuous. 

The coupling $(Z_n,Z)$, where $Z_n\stackrel{d}{=}Z_1(\floor{n\cdot})/b_n$,  is crucial in the proof of Theorem~\ref{septupleLimitTheorem}.

\begin{proof}[Proof of Theorem~\ref{septupleLimitTheorem}]
The coupling constructed above implies $(X_n, Y_n,C_n,X_n\circ C_n,Z_n)\to(X,Y,C,X\circ C,Z)$ almost surely as $n\to\infty$.
It remains to justify the convergence of the local time components. This is where we can no longer assert the almost sure convergence and have to resort to convergence in probability. 
In the case of $L_1(X_1)$, this follows from~\cite[Thm~2]{MR4463082}, 
except that we do not immediately get that the scaling sequence is regularly varying. 
Indeed, the scaling sequence $(a_n)$ obtained in \cite{MR4463082}, 
for which $L_n(t):= L_1(X_1)(\floor{nt})/a_n$ converges as $n\to \infty$ to $L(X)$,  
is expressed in terms of the excursion measure of $X_n$. 
The regular variation of $(a_n)$ is established in  Appendix~\ref{subce:a_n_reg_var} below. 
 
The scaling limit of $L_1(Z_1)$ is, of course, 
the content of Theorem \ref{theorem_Local_time_conv} and follows from our assumption, which in turn imply the assumptions of Theorem~\ref{thmLocalTimeLimitTheoremFromMUB22} above. 
Since the above limit theorems all hold in probability whenever the random walk components converge almost surely, the stated joint weak convergence follows. 
\end{proof}

\subsection{Organisation of the remainder of the paper}
In Section \ref{sectionProofOverview}, 
we discuss the proof of Theorem \ref{theorem_Local_time_conv} 
in the context of the framework of Theorem \ref{theorem_MUB22}. 
Section~\ref{sectionProofOverview} gives a proof of Theorem~\ref{theorem_Local_time_conv} using a result on branching processes, Theorem~\ref{propositionConvergenceOfPairsdepsdgepsg}, stated and applied at the end of the section. 
Finally, in Section \ref{sectionProofsOfAuxiliaryResults}, 
we prove Theorem~\ref{propositionConvergenceOfPairsdepsdgepsg}. This proof  is of independent interest as it  extends  classical results on (local) extinction times of branching processes to the setting of infinite variance and establishes a new Yaglom limit for BGWIs stated in Theorem~\ref{lemmaYaglomLimit} above. 
The proof of Theorem~\ref{propositionConvergenceOfPairsdepsdgepsg} relies on a web of results, discussed briefly in the first paragraph of Section~\ref{sectionProofsOfAuxiliaryResults} and 
pictorially represented in the diagram following it. Appendix~\ref{appendix_aroundRegularVariation} discusses Assumption~\ref{assumption_SL}. Appendix~\ref{app:example_non_conv_local_time} illustrates the fragility of the convergence of local time in the context of critical BGWIs.

\section{Proof of Theorem~\ref{theorem_Local_time_conv}}
\label{sectionProofOverview}

The of our main result in Theorem~\ref{theorem_Local_time_conv} consists of verifying the assumptions in equation~\eqref{eq_theorem_MUB22_a_assumption} of Theorem~\ref{theorem_MUB22}. 
In the applications of Theorem~\ref{theorem_MUB22} in~\cite{MR4463082},  its assumptions were verified using path-wise arguments. 
Indeed, we see that~\eqref{eq_theorem_MUB22_a_assumption} mostly involves the convergence of the hitting times of zero, both
forwards and backwards in time from a given time $t$. 
However, the path-wise arguments of~\cite{MR4463082} do not suffice in the setting of 
Theorem~\ref{theorem_Local_time_conv} 
because, for the CBI, we wish to analyse the hitting times of the boundary point $0$, which is 
 difficult to detect from the hitting times of the approximating BGWI processes (see example in 
Appendix~\ref{app:example_non_conv_local_time} below where BGWIs converge to the CBI $Z$ but the counting local times at $0$ do not converge to the Markov local time of $Z$ at $0$).

The innovation in the proof of Theorem~\ref{theorem_Local_time_conv} 
comes from introducing a new technique for establishing the limits in probability in~\eqref{eq_theorem_MUB22_a_assumption}, where
$Z_n$ and $Z$ are coupled as in Subsection~\ref{subsec:septuple_lim_thm} above.
Unlike the applications in~\cite{MR4463082}, our new technique relies on  a mixture of path-wise and distributional arguments. In  order to prove the assumptions in~\eqref{eq_theorem_MUB22_a_assumption} involving the limits of  starts and ends of excursions $g_t$ and $d_t$ 
defined in~\eqref{eq_thmLocalTimeLimitTheoremFromMUB22_g_t_d_t}, we introduce analogous quantities (both for $Z$ and $Z_n$) at any positive level $\eps>0$:
\begin{equation}
\label{eq_thmLocalTimeLimitTheoremFromMUB22_geps_t_deps_t}
    g^\eps_t(Z)\coloneqq \sup\set{s\leq t: Z(t)\in [0,\eps)}
    \quad\text{and}\quad
    d^\eps_t(Z)\coloneqq \inf\set{s> t: Z(t)\in [0,\eps)}. 
\end{equation}
with conventions $\sup\emptyset=0$ and $\inf\emptyset=\infty$. 

The following inequalities are valid for any $t\geq 0$ and any $\eps,\eta>0$:
\[
   \proba{\abs{d_t(Z)-d_t(Z_n)}>\eta}\leq \mathrm{I}_d+\mathrm{II}_d+\mathrm{III}_d
   \quad\text{and} 
   \quad\proba{\abs{g_t(Z)-g_t(Z_n)}>\eta}\leq \mathrm{I}_g+\mathrm{II}_g+\mathrm{III}_g,
\]
where the summands on the right-hand sides of the two inequalities are given by
\begin{equation}
\label{eq:def_I_II_III}
\begin{array}{lcl}
 \mathrm{I}_d
    \coloneqq\proba{\abs{d_t(Z)-d^{\eps}_{t}(Z)}>\eta/3},& &\mathrm{I}_g
    \coloneqq\proba{\abs{g_t(Z)-g^{\eps}_{t}(Z)}>\eta/3},\\
    \mathrm{II}_d
    \coloneqq\proba{\abs{d^\eps_{t}(Z)-d^\eps_{t}(Z_n)}>\eta/3},& & \mathrm{II}_g
    \coloneqq\proba{\abs{g^\eps_{t}(Z)-g^\eps_{t}(Z_n)}>\eta/3},\\
    \mathrm{III}_d
    \coloneqq\proba{\abs{d_t(Z_n)-d^{\eps}_{t}(Z_n)}>\eta/3},& &\mathrm{III}_g
    \coloneqq\proba{\abs{g_t(Z_n)-g^{\eps}_{t}(Z_n)}>\eta/3}. 
    \end{array}\end{equation}
By the coupling $(Z_n,Z)$, constructed before the proof of Theorem~\ref{septupleLimitTheorem} above, we may assume that 
the scaled BGWI processes $Z_n\stackrel{d}{=}Z_1(\floor{n\cdot})/b_n$ converge to the CBI process $Z$ almost surely $Z_n \stackrel{\text{a.s.
}}{\to} Z$ as $n\to\infty$.\footnote{The properties of the coupling $(Z_n,Z)$ used in the remainder of the paper are the almost sure convergence and the correct distributions of the processes $Z_n$ and $Z$. The processes $X_n,Y_n,C_n,X,Y,C$, used in this coupling construction, do not feature in the proof of Theorem~\ref{theorem_Local_time_conv}.}
Since this implies convergence in probability $Z_n \stackrel{\mathbb{P}}{\to} Z$ as $n\to\infty$,  Theorem~\ref{theorem_Local_time_conv} will follow from  Theorem~\ref{theorem_MUB22} if we can prove that, 
for every $\eta>0$, 
we have 
\begin{equation}
\label{eq:lim}
\lim_{\eps\to 0}\limsup_{n\to\infty}P(n,\eps)=0,\quad\text{ where }\quad 
P(n,\eps)\in\{\mathrm{I}_d, \mathrm{II}_d, \mathrm{III}_d, \mathrm{I}_g, \mathrm{II}_g, \mathrm{III}_g\}.
\end{equation}

The remainder of the paper analyses each of these limits. 
We first deal with the limits of the probabilities $\mathrm{I}_d,\mathrm{I}_g$, which involve the limiting process only and essentially only require the quasi left continuity of $Z$. We then analyse the limits of $\mathrm{II}_d,\mathrm{II}_g$,   requiring the coupling $(Z_n,Z)$ mentioned above and a general path-wise convergence result in Lemma~\ref{lemmaSkorokhodSpaceHittingTimeOfOpenSet}, downwards regularity of $Z$ and its quasi left continuity.  
The limits in~\eqref{eq:lim} of probabilities $\mathrm{III}_d,\mathrm{III}_g$
follow from Theorem~\ref{propositionConvergenceOfPairsdepsdgepsg} below, whose proof in Section~\ref{sectionProofsOfAuxiliaryResults} below  requires both weak convergence of the hitting times of branching processes as well as path-wise arguments,  including the aforementioned Skorokhod convergence result in Lemma~\ref{lemmaSkorokhodSpaceHittingTimeOfOpenSet}.
 
\begin{proof}[Proof of Theorem \ref{theorem_Local_time_conv} (case $\lim_{\eps\to0}\mathrm{I}_{d,g}=0$ for $\mathrm{I}_{d,g}$ in~\eqref{eq:def_I_II_III})] 
Note  that by~\eqref{equationLaplaceTransformOfCBI} the following limit holds $\proba{Z(t)=0}=
\lim_{\lambda\to\infty}\imf{\mathbb{E}_z}{e^{-\lambda Z(t)}}=0$ 
for all $Z(0)=z\in[0,\infty)$ and $t>0$.
Moreover, since $\delta\in (0,1)$,
 \cite[\S 5.2.1]{MR3263091} implies that $Z$ is recurrent. 
As neither $\mathrm{I}_d$ nor $\mathrm{I}_g$ depend on $Z_n$, 
we will in this part of the proof temporarily denote $d_t=d_t(Z)$ and $d_t^\eps=d_t^\eps(Z)$ and, similarly,  $g_t=g_t(Z)$ and $g_t^\eps=g_t^\eps(Z)$. Pick $t>0$.

Since $Z$ has no negative jumps,  for $0< \eps<Z(t)$ we get $d^{\eps}_t\leq d_t<\infty$ and $Z(d^\eps_t)=\eps$.  Moreover, $d^\eps_t$ increases strictly, 
say to $\tilde d_t\leq d_t$, as $\eps\to 0$. 
Recall that the self similar CBI process $Z$ of Theorem \ref{theorem_Local_time_conv}  
is Feller  (evident from the one-dimensional distributions in \eqref{equationLaplaceTransformOfCBI}) and therefore quasi left continuous, 
implying $\eps= Z({d^\eps_t})\to Z({\tilde d_t})$ 
so that $Z(\tilde d_t)=0$ and hence $\tilde d_t=d_t$. 
Hence $d^\eps_t\to d_t$ as $\eps\to 0$ a.s. for all $t>0$, implying $\mathrm{I}_d\to0$ as $\eps\to0$.

Regarding $\mathrm{I}_g$, 
on the interval $(g_t,d_t)$,   $g^{\eps}_t$ decreases (as $\eps\to0$) to a limit we will denote $\tilde g_t\geq g_t$. 
Note that, by definition $\eps\geq Z({g^\eps_t-})\to 0$  as $\eps\to 0$ ($Z$ has left limits $Z(s-)$, $s\geq0$, with $Z(0-)\coloneqq 0$).
If $g^\eps_t>\tilde g_t$ for all small $\eps>0$ then $0=\lim_{\eps\to 0} Z({g^\eps_t-})=Z({\tilde g_t})$ by right continuity, 
so that $Z({\tilde g_t})=0$ and therefore $\tilde g_t=g_t$. 
If $g^\eps_t=\tilde g_t$ for all $\eps$ small enough, it means that $Z$ jumps at $\tilde g_t$ 
and, by the lack of negative jumps, $Z({\tilde g_t})>0$. Since $\eps\geq Z({g^\eps_t-})= Z({\tilde g_t-})$ for all small $\eps>0$, we get $Z({\tilde g_t-})=0$. 
To conclude, 
we will use the following claim, 
found in \cite[p. 44]{MR4463082} and which essentially uses quasi left continuity: 
during each excursion interval $(g,d)$ of $Z$ away from $0$ 
(i.e. a connected component of the excursion set
   $ [0,\infty)\setminus\overline{ \{s\geq 0: Z(s)=0\} }$),
for every $u\in (g,d)$, 
we have not only that $Z(u)\neq 0$ but also that $Z(u-)\neq 0$. 
(We prove a similar statement at level $\eps$ in our analysis of $\mathrm{II}_d$ below.) 
The claim  implies that $\tilde g_t$ does not belong to the interval $(g_t,t]\subset(g_t,d_t)$ 
and therefore $\tilde g_t=g_t$. 
Thus $\lim_{\eps\to 0}g^\eps_t=g_t$ a.s., implying $\mathrm{I}_g\to0$ as $\eps\to0$.
\end{proof}

The path-wise arguments for probabilities $\mathrm{II}_{d}, \mathrm{II}_{g}$ in~\eqref{eq:def_I_II_III} rely on a Skorokhod space result below, stated for a metric space 
$(E,d)$. Let $D$ denote the Skorokhod space of \cadlag\  functions from $[0,\infty)$ into $E$ (i.e. $f\in D$ is right continuous 
and has left limits).
Recall that if a sequence $(f_n)_{n\in\na}$ tends to $f$ in the Skorokhod $J_1$-topology on $D$, for every $T>0$ at which $f$ is continuous 
there exists a sequence of increasing homeomorphisms $(\lambda_n)_{n\in\na}$ from $[0,T]$ to itself, satisfying  
$\sup_{t\in[0,T]}\max\{|\lambda_n(t)-t|,d(f_n(t),f(\lambda_n(t))\}\to0$ as $n\to\infty$,
see~\cite[Sec.~5]{MR838085} for details. In particular, for any sequence $t_n\to t\in [0,T)$, where $f$ is continuous at the time $t$, we have 
\begin{equation}
\label{eq:uniform_convergence}
d(f_n(t_n),f(t))\leq d(f_n(t_n),f(\lambda_n(t_n)))+d(f(\lambda_n(t_n)),f(t))\to0\quad\text{as $n\to\infty$.}
\end{equation}

Using the convention $\inf\emptyset=\infty$, for any set $U\subset E$, define the hitting time (resp. left-limit hitting time)  of $U$ by a function $f\in D$ as follows:
\begin{equation}
\label{eq:def_T_U(f)}
\imf{T_U}{f}\coloneqq \inf\set{t\geq 0: \imf{f}{t}\in U}\qquad\text{(resp. 
 }\>\imf{T^-_U}{f}\coloneqq\inf\set{t\geq 0: \imf{f}{t-}\in U}\text{).}
\end{equation}
Here and throughout we define the left limit at $t>0$ by $f(t-)\coloneqq \lim_{s\uparrow t}f(t)$  and at $t=0$ by $f(0-)\coloneqq f(0)$.
It follows easily from the definition in~\eqref{eq:def_T_U(f)} that for any
open set $O$ in $E$ the following holds: if $t$ satisfies $f(t)\in O$ (resp. $f(t-)\in O$), then $T_O^-(f)\leq t$ (resp. $T_O(f)\leq t$). Thus any 
\cadlag\ function $f$ satisfies 
\begin{equation}
    \label{eq:T=T^-}
    T_O(f)=T_O^-(f)\qquad\text{for any open set $O$ in $E$.}
\end{equation}

\begin{lemma}
\label{lemmaSkorokhodSpaceHittingTimeOfOpenSet}
Let $O\subset E$ be open and denote its closure in $E$ by $C$.
Let $f\in D$ satisfy
\begin{equation}
\label{eq:Skorkhod_lem_ass1}
\imf{T_O}{f}\leq \min\{\imf{T_C}{f} , \imf{T^-_{C}}{f}\}. \end{equation}
If   a sequence $\paren{f_n}_{n\in\na}$ in $D$  converges in $J_1$-topology to $f$, then
$\imf{T_O}{f_n}\to \imf{T_O}{f}$ as $n\to\infty$ in $[0,\infty]$. 
If, in addition, 
we assume
\begin{equation}
\label{eq:Skorkhod_lem_ass2}
T_O(f)<\infty\qquad\text{and}\qquad    f(T_O(f))=f(T_O(f)-), 
\end{equation}
then $\imf{T_O}{f_n}<\infty$ for all large $n$ and
$ \paren{\imf{T_O}{f_n},\imf{f_n}{\imf{T_O}{f_n}}}\to \paren{\imf{T_O}{f},\imf{f}{\imf{T_O}{f}}}$ as $n\to\infty$ in $[0,\infty)\times E$
\end{lemma}

\begin{proof}
Under assumption~\eqref{eq:Skorkhod_lem_ass2}, $f$ is continuous at 
$\imf{T_O}{f}<\infty$. Thus, by~\eqref{eq:uniform_convergence}, the second assertion of the lemma follows from the first. We now prove that~\eqref{eq:Skorkhod_lem_ass1} implies $\imf{T_O}{f_n}\to \imf{T_O}{f}$.

We first show $\limsup_{n} \imf{T_O}{f_n}\leq \imf{T_O}{f}$. If $\imf{T_O}{f}=\infty$, this is obvious.
If $\imf{T_O}{f}<\infty$,
for any $\delta>0$, there exists a continuity point  $t'\in[\imf{T_O}{f},\imf{T_O}{f}+\delta]$ of $f$  
such that $\imf{f}{t'}\in O$.
By~\eqref{eq:uniform_convergence} (with $t_n=t'$ for all $n\in\na$) we have $\imf{f_n}{t'}\in O$ for all large  $n$, implying
 the inequality. Moreover,  since $\liminf_n \imf{T_O}{f_n}\geq0$, if $\imf{T_O}{f}=0$ then the limit $\imf{T_O}{f_n}\to 0$ holds.

We assume $\imf{T_O}{f}\in(0,\infty]$ and have to prove $\liminf_n T_O(f_n)\geq T_O(f)$. By~\eqref{eq:Skorkhod_lem_ass1},
for any $s\in [0,\imf{T_O}{f})$,
we have $\imf{f}{s}\not\in C$ and $\imf{f}{s-}\not\in C$.
Fix arbitrary $T\in[0,\imf{T_O}{f})$.
Then we have 
\begin{equation}
\label{eq:def_of_eta}
   \eta:=\inf\{ d(f(s),C)\wedge d(f(s-),C):s\in[0,T]\}>0,
\end{equation}
where $u\wedge v:= \min\{u,v\}$,  $u,v\in\re$.
Indeed, if to the contrary $\eta=0$, there exists a sequence $(s_n)_{n\in\na}$ in $[0,T]$ converging to
$s'\in[0,T]$, such that $d(f(s_n),C)\wedge d(f(s_n-),C)\to0$ as $n\to\infty$. Thus 
there exists $(s_n')_{n\in\na}$ in $[0,T]$, such that  $s_n'\to s'$ and $d(f(s_n'),C)\to0$. Since $f$ is \cadlag,
$d(f(s'),C)\leq d(f(s_n'),C)+
d(f(s_n'),f(s'))$
and 
$d(f(s'-),C) \leq d(f(s_n'),C)+ d(f(s_n'),f(s'-))$,
we have
$$
d(f(s'),C)\wedge d(f(s'-),C)\leq 
d(f(s_n'),C)+
d(f(s_n'),f(s'))\wedge d(f(s_n'),f(s'-))\to0
$$
as $n\to\infty$, implying either $f(s')\in C$ or $f(s'-)\in C$ (recall $C$ is closed) and contradicting~\eqref{eq:Skorkhod_lem_ass1}.

We now show that~\eqref{eq:def_of_eta} implies the following claim: there exists $N\in\na$ such that  
\begin{equation}
\label{eq:uniform_bound_for_all_large_n}
    \inf\{ d(f_n(s),C)\wedge d(f_n(s-),C):s\in[0,T]\}>\eta/4\qquad\text{ for all $n\geq N$.}
\end{equation}
Assume claim~\eqref{eq:uniform_bound_for_all_large_n} is false. Then there exist sequences $(n_k)_{k\in\na}$ in $\na$, with $n_k\to\infty$, and $(s_k)_{k\in\na}$ in $[0,T]$, such that 
$d(f_{n_k}(s_k),C)\wedge d(f_{n_k}(s_k-),C)\leq 3\eta/8$ for all $k\in\na$. Hence there exists a sequence 
$(s_k')_{k\in\na}$, such that $d(f_{n_k}(s_k'),C)<\eta/2$ for all $k\in\na$.
Moreover, by passing to a subsequence, we may assume that $s_k'\to s'\in[0,T]$ as $k\to\infty$. Hence, for $k\in\na$, we have
\begin{align}
\label{eq:upper_bound_on_distance_from_C}
d(f(s'),C)\wedge d(f(s'-),C) & \leq d(f_{n_k}(s_k'),C)+d(f(s'),f_{n_k}(s_k'))\wedge d(f(s'-),f_{n_k}(s_k')).
\end{align}
Convergence $f_{n_k}\to f$ in $J_1$-topology on $D$  and 
a triangle inequality (analogous  to~\eqref{eq:uniform_convergence}, using $f(\lambda_k(s_k'))$ with  increasing homeomorphisms 
$\lambda_k;[0,T]\to[0,T]$, $k\in\na$)  
imply
\begin{align*}
d(f(s'),f_{n_k}(s_k'))\wedge d(f(s'-),f_{n_k}(s_k')) \leq &  d(f(s'),f(\lambda_k(s_k')))\wedge d(f(s'-),f(\lambda_k(s_k')))\\ 
& + \sup_{u\in[0,T]} d(f(\lambda_k(u)),f_{n_k}(u))\to0 \text{ as $k\to\infty$.}
\end{align*}
Thus by~\eqref{eq:upper_bound_on_distance_from_C}
we obtain 
$d(f(s'),C)\wedge d(f(s'-),C)\leq \eta/2$, contradicting~\eqref{eq:def_of_eta} since $s'\in[0,T]$.

By~\eqref{eq:uniform_bound_for_all_large_n} we get 
$T_O(f_n)\geq T$ for all $n\geq N$. Since $T\in[0,\imf{T_O}{f})$ was arbitrary, we have $\liminf_n T_O(f_n)\geq T_O(f)$, concluding the proof of the lemma.
\end{proof}

Recall that the coupling $(Z_n,Z)$, where $Z_n\stackrel{d}{=}Z_1(\floor{n\cdot})/b_n$ and $Z$ denotes the self-similar CBI in~\eqref{eq_LampertiTransformation}, 
is such that $Z_n\stackrel{\text{a.s.}}{\rightarrow} Z$ as $n\to\infty$.

\begin{proof}[Proof of 
Theorem \ref{theorem_Local_time_conv} 
(case $\lim_{\eps\to 0}\lim_{n\to\infty}\mathrm{II}_{d}=0$ for $\mathrm{II}_{d}$ in~\eqref{eq:def_I_II_III})]

Pick arbitrary $\eps>0$ and $t\geq0$ and define $O:=[0,\eps)\subset E:=[0,\infty)$ and its closure $C:=[0,\eps]$. Recall $d_t^\eps(Z)=\inf\{s> t: Z(s)\in O\}$ and $d_t^\eps(Z_n)=\inf\{s> t: Z_n(s)\in O\}$.
 Since $O$ is open, note that $T_O(f)$ (defined in~\eqref{eq:def_T_U(f)}
 for any \cadlag\ 
$f:[0,\infty)\to E$) satisfies 
$T_O(f)=\inf\{t>0:f(t)\in O\}$.
Define the \cadlag\ processes $Z^t:=Z(t+\cdot)$ and $Z_n^t:=Z_n(t+\cdot)$ with the same semigroups as $Z$ and $Z_n$, respectively. Then, by Lemma~\ref{lemmaSkorokhodSpaceHittingTimeOfOpenSet}, $d_t^\eps(Z)=t+T_O(Z^t)<\infty$ and $d_t^\eps(Z_n)=t+T_O(Z_n^t)<\infty$ satisfy   
$|d_t^\eps(Z)-d_t^\eps(Z_n)|\stackrel{n\to\infty}{\longrightarrow}0$ almost surely, thus implying  Theorem~\ref{theorem_Local_time_conv} in this case (since $\mathrm{II}_{d}=\proba{|d_t^\eps(Z)-d_t^\eps(Z_n)|>\eta/3}\to0$ as $n\to\infty$, the limit in~\eqref{eq:lim} holds for $\mathrm{II}_{d}$), if we establish
\begin{equation}
\label{eq:assum_Lem_2_IId}
    T_O(Z^t)\leq T_C(Z^t)\wedge T_C^-(Z^t)\qquad\text{almost surely.}
\end{equation}

We first prove $T_O(Z^t)= T_C(Z^t)$ a.s. Since $O\subset C$, we have
$T_O(Z^t)\geq T_C(Z^t)$ and, as $T_C(Z^t)$ is a stopping time,  also $\proba{T_C(Z^t)<T_O(Z^t)}=\imf{\mathbb{E}}{\imf{\p_{Z(T_C(Z^t))}}{0<T_O(Z)}}$
as well as $Z(T_C(Z^t))\in C$ a.s. If $Z(T_C(Z^t))\in O$, then $\p_{Z(T_C(Z^t))}(0<T_O(Z))=0$. If $Z(T_C(Z^t))=\eps$, then~\cite[p.~45]{MR4463082} (under the heading \emph{Regular and instantaneous character of $z$}) yields downwords regularity of $Z$ at any positive level, implying $\p_{\eps}(0<T_O(Z))=0$.
Hence $\proba{T_C(Z^t)<T_O(Z^t)}=0$ and the equality $T_O(Z^t)= T_C(Z^t)$ holds almost surely.

We now prove  $T_C(Z^t)=T^-_C(Z^t)$ a.s. 
Notice that the downwards regularity of $Z$ in the previous paragraph and the strong Markov property at $T_C(Z^t)$ yield
 $T^-_C(Z^t)\leq T_C(Z^t)$ a.s. 
Consider $C_\delta:=[0,\eps+\delta]$  and note that for $\delta>0$ we have $T_{C_\delta}(Z^t)\leq T^-_C(Z^t)\leq T_C(Z^t)$ a.s.
If $\delta'>\delta>0$, then $T_{C_{\delta'}}(Z^t)\leq T_{C_\delta}(Z^t)$ with strict inequality on the event $\{Z^t(T_{C_{\delta'}}(Z^t))>\delta\}$. Hence the stopping time $\tilde T:=\lim_{\delta\downarrow0}T_{C_\delta}(Z^t)$  satisfies
\begin{equation}
\label{eq:useful_inequality}
    \tilde T\leq T^-_C(Z^t)\leq T_C(Z^t)\quad\text{a.s.} 
\end{equation}
If $Z^t(0)\in C$, then by definition~\eqref{eq:def_T_U(f)} we have $T_C(Z^t)=0$. Moreover, in this case we have $T_{C_\delta}(Z^t)=0$ for all $\delta>0$, implying 
$\tilde T=0$ and, by~\eqref{eq:useful_inequality}, the equality $T^-_C(Z^t)= T_C(Z^t)$. 
If $Z^t(0)\notin C$, then $Z^t(T_{C_\delta}(Z^t))>\eps$ for all small $\delta>0$ and
 $T_{C_\delta}(Z^t)\uparrow\tilde T$ as $\delta\downarrow0$. The quasi left continuity of $Z$ at $\tilde T$ implies
$Z^t(\tilde T)=\lim_{\delta\downarrow0}
Z^t(T_{C_\delta}(Z^t))=\eps\in C$, where the last equality follows from $Z^t(T_{C_\delta}(Z^t))=\eps+\delta$ for all small $\delta>0$ (which holds since $Z$ has no negative jumps). In particular we get $\tilde T\geq T_C(Z^t)$. By~\eqref{eq:useful_inequality} we obtain
$\tilde T= T^-_C(Z^t)= T_C(Z^t)$ a.s. 
Thus~\eqref{eq:assum_Lem_2_IId} holds.
\end{proof}


\begin{proof}[Proof of 
Theorem \ref{theorem_Local_time_conv}
(case $\lim_{\eps\to 0}\lim_{n\to\infty}\mathrm{II}_{g}=0$ for $\mathrm{II}_{g}$ in~\eqref{eq:def_I_II_III})] 
Let $O=[0,\eps)$, $C=[0,\eps]$ and define the \cadlag\ process $\hat Z$ by $\hat Z(s)\coloneqq Z(\max\{t-s,0\}-)$ for  $s\in[0,\infty)$. 
Note that $\hat Z$ runs along the path of $Z$ backward from time $t$ to $0$ and for $s\in[t,\infty)$ we have $\hat Z(s)= Z({0-})=Z(0)=0$. By the definitions in~\eqref{eq_thmLocalTimeLimitTheoremFromMUB22_geps_t_deps_t} of $g_t^\eps(Z)$
and in~\eqref{eq:def_T_U(f)} of $T_O^-(\hat Z)$ we obtain
\begin{align*}
g_t^\eps(Z) & =\sup\{s\in[0,t]:Z(s)\in O\}=\sup\{t-s\in[0,t]:Z(t-s)\in O\}\\
& =t-\inf\{s\in[0,t]:Z(t-s)\in O\}\\
& =t-\inf\{s\in[0,\infty):Z(\max\{t-s,0\})\in O\}
\\
&=t-\inf\{s\in[0,\infty):\hat Z(s-)\in O\}=t-T_O^-(\hat Z)=t-T_O(\hat Z),
\end{align*}
where the last equality holds by~\eqref{eq:T=T^-} since $O$ is open in $E=[0,\infty)$ and $\hat Z$ is \cadlag.
The same argument applied to the \cadlag\ process $\hat Z_n$, where $\hat Z_n(s)\coloneqq Z_n(\max\{t-s,0\}-)$ for $s\in[0,\infty)$, yields
$g_t^\eps(Z_n)=t-T_O(\hat Z_n)$.

Given the representations  
$g_t^\eps(Z)=t-T_O(\hat Z)$ and
$g_t^\eps(Z_n)=t-T_O(\hat Z_n)$,
as in the case $\mathrm{II}_{d}$, the aim is to verify Assumption~\eqref{eq:Skorkhod_lem_ass1} of Lemma~\ref{lemmaSkorokhodSpaceHittingTimeOfOpenSet},  which 
(phrased in terms of $\hat Z$) requires 
\begin{equation}
\label{eq:assumption_Lem_4_for_Z}
    T_O(\hat Z)\leq \min\{T_C(\hat Z) , T^-_{C}(\hat Z)\}.
\end{equation}
Then, by Lemma~\ref{lemmaSkorokhodSpaceHittingTimeOfOpenSet},
we have
$T_O(\hat Z_n)\to T_O(\hat Z)$ almost surely as $n\to\infty$, implying 
by definition~\eqref{eq:def_I_II_III} the limit
$\lim_{n\to\infty}\mathrm{II}_g=0$ for every $\eps>0$
and concluding the proof of this case of Theorem~\ref{theorem_Local_time_conv}.

It remains to prove the inequality in~\eqref{eq:assumption_Lem_4_for_Z}.
The key difference with the case $\mathrm{II}_{d}$ is that the dynamics of $\hat Z$ is not tractable, making its regularity at $\eps$ hard to establish. 
The solution is to rephrase~\eqref{eq:assumption_Lem_4_for_Z} in terms of the forward process $Z$ and its behaviour upon touching or approaching level $\eps$. 
Note first that $\hat Z$ only has negative jumps and hence
$T_C(\hat Z)\leq T^-_C(\hat Z)$. 
Indeed, if $s$ is such that $\hat Z(s-)\leq \eps$, 
then $\hat Z(s)=\hat Z(s-)+(\hat Z(s)-\hat Z(s-))\leq \eps$. 
Hence~\eqref{eq:assumption_Lem_4_for_Z} holds if we show $T_O(\hat Z)\leq T_C(\hat Z)$ a.s.

Recall that by definition $T_O(\hat Z)\leq t$, since $\hat Z(t)=Z(0)=0$. On the event $\{T_C(\hat Z)<T_O(\hat Z)\}$, the forward process $Z$ behaves as follows at some time $s\in[0,t]$:   \textit{either} $Z$ approaches $\eps$ from above by a left limit at $s$ and jumps up at $s$ and remains above $\eps$ on $[s,t]$ \textit{or} $Z$ touches $\eps$ from above continuously at $s$ and remains above $\eps$ on $(s,t]$. 
Hence
$$\{T_C(\hat Z)<T_O(\hat Z)\}\subset \cup_{r\in\mathbb{Q}\cap[0,t]}\{T_C(Z^r)<T_O(Z^r)\},$$
where, as shown in
the proof of $\mathrm{II}_d$ above, we have
$\p(T_C(Z^r)<T_O(Z^r))=0$ for all $r\in\mathbb{Q}\cap[0,t]$ (recall $Z^r=Z(r+\cdot)$).
Therefor $T_O(\hat Z)\leq T_C(\hat Z)\leq T^-_C(\hat Z)$ and~\eqref{eq:assumption_Lem_4_for_Z} holds.
\end{proof} 

We now proceed with the proof of the limits of the probabilities $\mathrm{III}_d$ and $\mathrm{III}_g$ in~\eqref{eq:def_I_II_III}. It is based on  Theorem~\ref{propositionConvergenceOfPairsdepsdgepsg}, which will be established in Section~\ref{sectionProofsOfAuxiliaryResults} below via uniform distributional control of forward-looking hitting times of zeros of BGWI processes.

\begin{theorem}
\label{propositionConvergenceOfPairsdepsdgepsg}
Under Assumption \ref{assumption_SL} with $\delta=\frac{d}{\alpha c}\in(0,1)$, for any $t\geq0$ and $\eps>0$, we have 
\[
    \paren{g^\eps_t(Z_n),g_t(Z_n)}
    \stackrel{d}{\to}\paren{g^\eps_t(Z),g_t(Z)}
    \quad\text{and}\quad
    \paren{d^\eps_t(Z_n),d_t(Z_n)}
    \stackrel{d}{\to}\paren{d^\eps_t(Z),d_t(Z)}\quad\text{as $n\to\infty$.}
\]
\end{theorem}

Before concluding the proof of Theorem~\ref{theorem_Local_time_conv}, 
note that, as seen in the proof above of the limits of the probabilities $\mathrm{II}_d$ and $\mathrm{II}_g$ in~\eqref{eq:def_I_II_III}, 
under our coupling we in fact have almost sure convergence of the first components
in the both weak limits  of Theorem~\ref{propositionConvergenceOfPairsdepsdgepsg}.

\begin{proof}[Proof of Theorem \ref{theorem_Local_time_conv}: $\lim_{\eps\to 0}\lim_{n\to\infty}\mathrm{III}_{g,d}=0$ for $\mathrm{III}_{d}$ in~\eqref{eq:def_I_II_III}]
Theorem~\ref{propositionConvergenceOfPairsdepsdgepsg} implies the weak convergence
\[
    g^\eps_t(Z_n)-g_t(Z_n)
    \stackrel{d}{\to }g^\eps_t(Z)-g_t(Z)
    \quad\text{and}\quad
    d_t(Z_n)-d^\eps_t(Z_n)
    \stackrel{d}{\to }d_t(Z)-d^\eps_t(Z)\quad\text{as $n\to\infty$.}
\]Therefore, upon taking the limit as $n\to\infty$ of probabilities $\mathrm{III}_d$ and $\mathrm{III}_g$ in~\eqref{eq:def_I_II_III}, 
we obtain $\mathrm{I}_d$ and $\mathrm{I}_g$. 
By the preceding proof of the cases $\mathrm{I}_d$ and $\mathrm{I}_g$,  we can then take the limit as $\eps\to 0$ to get $0$, concluding the proof of Theorem~\ref{theorem_Local_time_conv}.
\end{proof}

\section{Convergence of the hitting times of BGW and BGWI: proofs of the Yaglom limit and Theorem~\ref{propositionConvergenceOfPairsdepsdgepsg}}

\label{sectionProofsOfAuxiliaryResults} 

To conclude the proof of Theorem \ref{theorem_Local_time_conv}, it remains to establish Theorem~\ref{propositionConvergenceOfPairsdepsdgepsg}.
The proof is a succession of steps which at first glance appear unrelated. We first prove the convergence of the extinction times of BGW process (Lemma~\ref{lemmaConvergenceOfExtinctionTimesOfBGW} below). We then establish the asymptotics of the probability $\proba{Z_1(n)=0}$ for the BGWI process $Z_1$ and apply Tauberian theorems to determine the asymptotics for the tails of the first return time to zero
$\proba{d_1(Z_1)>n}$
for BGWI processes
$Z_1$, see Lemma~\ref{lemmaPreLimitAsymptotics} below. We use the asymptotics of 
$\proba{Z_1(n)=0}$ and $\proba{d_1(Z_1)>n}$
 to prove a local limit theorem for $g_t(Z_n)$ (with limit $g_t(Z)$), cf. Corollary~\ref{corPreLimitAsymptotics} below. 
Using this local limit theorem and the extinction time convergence of the BGW processes, we prove the limit of the hitting times of zero for BGWI  started away from zero (cf. Lemma~\ref{propHittingTimeConvergence_varying_Staring_points} below) and then establish the weak limit of the pair $(d^\eps_t(Z_n),d_t(Z_n))$. 
The limit of $(g^\eps_t(Z_n),g_t(Z_n))$
follows from the limit in Lemma~\ref{prop:g_t,Z(t)-limit}, which in turn requires a general fact relating excursions exceeding a given length to the excursion straddling time $t$ (see~\cite[Thm~7.35]{MR525052}), the Yaglom limit for BGWIs in Theorem~\ref{lemmaYaglomLimit} (stated above)  and a characterisation of the conditional law of $Z(t)$, given $g_t(Z)=s$, in Proposition~\ref{lemmaCBIExcursionStraddlingt} below. 
It is crucial for our proof of Thereom~\ref{propositionConvergenceOfPairsdepsdgepsg} (and hence that of Theorem~\ref{theorem_Local_time_conv}) that this conditional law coincides with the limit law in the Yaglom limit of Theorem~\ref{lemmaYaglomLimit}.
The overview of the steps in the proof of Theorem~\ref{propositionConvergenceOfPairsdepsdgepsg} is given in the diagram:

\medskip
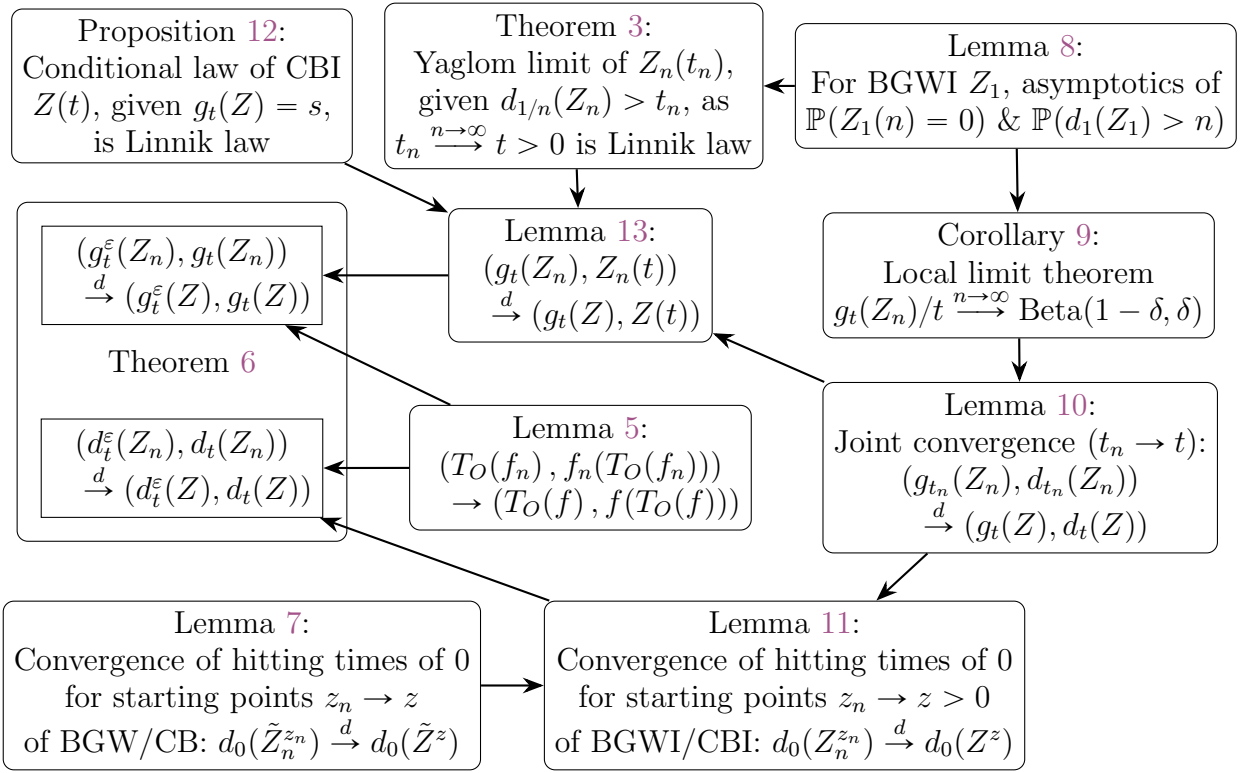
\begin{tikzpicture}[
    node distance=3em and 2em,
    block/.style={draw, rectangle, rounded corners, minimum width=3em, minimum height=2em, align=center},
    arrow/.style={-{Stealth[scale=1.2]}, thick},
    subblock/.style={draw, rectangle, minimum width=2em, minimum height=1em, align=center},
    container/.style={draw, rectangle, rounded corners, inner sep=0.75em}
]


\node[subblock] (D) {
$\paren{d^\eps_t(Z_n),d_t(Z_n)}$\\
\quad $\stackrel{d}{\to}\paren{d^\eps_t(Z),d_t(Z)}$
};
\node[subblock, above=of D] (G) { 
$\paren{g^\eps_t(Z_n),g_t(Z_n)}$\\
\quad $\stackrel{d}{\to}\paren{g^\eps_t(Z),g_t(Z)}$ 
};

\node[container, fit=(G)(D)] (T) {Theorem~\ref{propositionConvergenceOfPairsdepsdgepsg}};

\node[block, above=of G, yshift=-1em] (P12) {
Proposition~\ref{lemmaCBIExcursionStraddlingt}:\\
Conditional law of  CBI \\ 
$Z(t)$, given $g_t(Z) = s$,\\
is Linnik law
};

\node[block, right=of G, xshift=2em] (L13) {
Lemma~\ref{prop:g_t,Z(t)-limit}:\\
$(g_t(Z_n), Z_n(t))$\\
\quad $\stackrel{d}{\to} (g_t(Z), Z(t))$
};

\node[block, below=of L13, yshift=1em] (L5) {
Lemma~\ref{lemmaSkorokhodSpaceHittingTimeOfOpenSet}:\\
$ \paren{\imf{T_O}{f_n},\imf{f_n}{\imf{T_O}{f_n}}}$\\
\quad $\to \paren{\imf{T_O}{f},\imf{f}{\imf{T_O}{f}}}$
};

\node[block, right=of L5, xshift=0.25em] (L10) {
Lemma~\ref{propHittingTimeConvergence}:\\
Joint convergence ($t_n\to t$):\\
$(g_{t_n}(Z_n),d_{t_n}(Z_n))$\\
\quad $\stackrel{d}{\to} (g_t(Z),d_t(Z))$
};

\node[block, below=of L5, yshift=0.75em, xshift=6.5em] (L11) {
Lemma~\ref{propHittingTimeConvergence_varying_Staring_points}:\\
Convergence of hitting times of $0$ \\
for starting points  $z_n\to z>0$\\
of BGWI/CBI: $d_0(Z_n^{z_n})\stackrel{d}{\to} d_0(Z^z)$
};

\node[block, right=of P12, xshift=-1em] (T3) {
Theorem~\ref{lemmaYaglomLimit}: \\
Yaglom limit of $Z_n(t_n)$, \\
given $d_{1/n}(Z_n) > t_n$, as \\ 
$t_n\stackrel{n\to\infty}{\longrightarrow} t>0$ is Linnik law
};

\node[block, right=of L13, xshift=1.5em] (C9) {
Corollary~\ref{corPreLimitAsymptotics}:\\
Local limit theorem\\
$g_t(Z_n)/t\stackrel{n\to\infty}{\longrightarrow} \mathrm{Beta}(1-\delta,\delta)$
};

\node[block, left=of L11] (L7) {
Lemma~\ref{lemmaConvergenceOfExtinctionTimesOfBGW}:\\
Convergence of hitting times of $0$\\
for starting points $z_n\to z$ \\
of BGW/CB: $d_0(\tilde Z^{z_n}_n)\stackrel{d}{\to} d_0(\tilde Z^z)$
};

\node[block, right=of T3, xshift=-1em] (L8) {
Lemma~\ref{lemmaPreLimitAsymptotics}: \\
For BGWI $Z_1$, asymptotics of \\
$\proba{Z_1(n)=0}$ \&
$\proba{d_1(Z_1)>n}$
};

\draw[arrow] (L11) -- (D);
\draw[arrow] (L5) -- (D);
\draw[arrow] (L5) -- (G);
\draw[arrow] (L13) -- (G);
\draw[arrow] (L10) -- (L11);
\draw[arrow] (L10) -- (L13);
\draw[arrow] (T3) -- (L13);
\draw[arrow] (P12) -- (L13);
\draw[arrow] (L7) -- (L11);
\draw[arrow] (L8) -- (C9);
\draw[arrow] (L8) -- (T3);
\draw[arrow] (C9) -- (L10);
\end{tikzpicture}

\subsection{The local limit theorem for \texorpdfstring{$g_t(Z_n)/t$}{gt(Zn)/t} and the proof of the Yaglom limit}
Given our offspring distribution $\mu$ with generating function $f$ satisfying the domain of attraction condition \ref{assumption_SL}, let $\tilde Z_1$ be a BGW process whose reproduction dynamics are governed by $\mu$. 
This can be explicitly constructed using the random walk $X_1$ (with jump distribution given by the shifted offspring distribution $\tilde \mu$) by solving recursion \eqref{eq_discreteLampertiTransformation_1} with $Y_1\equiv0$ and a starting state $l\in\na$: 
\[
    \tilde Z_1^l= l+X_1\circ \tilde C_1\quad\text{where}\quad \tilde C_1(k)=\sum_{0\leq j<k} \tilde Z^l_1(j). 
\]We will denote the law of $\tilde Z_1^l$ by $\tilde\p_l$. The branching property of $\tilde Z_1^l$  is as follows,
\[
    \tilde \p_{l_1+l_2}=\tilde \p_{l_1}*\tilde \p_{l_2}\qquad\text{for any $l_1,l_2\in\na$,}
\]where $\tilde \p_{l_1}*\tilde \p_{l_2}$ denotes the law of the sum of independent  copies of 
$\tilde Z^{l_1}_1$ and $\tilde Z^{l_2}_1$. 
Recall also that the law of $\tilde Z^l_1(n)$ is characterised in terms of the $n$-th iteration $f^{\circ n}$ of the function $f$ in Assumption~\ref{assumption_SL} with itself (for $n=0$, we define $f^{\circ 0}(s):=s$)  for all $n, l\in\na$ as follows: 
\begin{equation}
    \label{eq:BGW_transform}
    \mathbb{E}(s^{\tilde Z_1^l(n)})=f^{\circ n}(s)^l.
\end{equation}
Since $0$ is an absorbing state for $\tilde Z_1^l$ (as no immigration is present), 
the hitting time of $0$ of $\tilde Z_1^l$, which equals $d_0(\tilde Z_1^l)$, 
then satisfies
\begin{equation}
\label{eq:BGW_extiction}
    \proba{d_0(\tilde Z_1^l)\leq n}=\proba{\tilde Z^l_1(n)=0}=f^{\circ n}(0)^l,\quad\text{for all $n,l\in\na$.} 
\end{equation}
Analogous to~\eqref{eq_discreteLampertiTransformation_1}, by e.g.~\cite{MR3098685}, we can define the CB process $\tilde Z^z$ started at $z\in[0,\infty)$ as the unique solution of
\[
    \tilde Z^z(t)=z+\imf{X}{\int_0^t \tilde Z^z(s)\, ds}.
\]
By~\cite[Thm~2.1]{MR2225068}, $\tilde Z^z$ is the (large population) scaling limit of $\tilde Z^{z_n}_n$, where $\tilde Z^{z_n}_n(t)=\tilde Z^{l_n}_1(\floor{nt})/b_n$, the scaling sequence $b_n$ satisfies equation \eqref{equationDefbn},    $l_n\coloneqq \floor{b_n z}$ and hence $z_n
\coloneqq l_n/b_n\to z$ as $n\to\infty$.
It is perhaps no surprise (although we cannot find it in the literature) that the extinction times (i.e. hitting times of zero) of $\tilde Z^{z_n}$ also converge weakly to those of $\tilde Z^z$.
\begin{lemma}
\label{lemmaConvergenceOfExtinctionTimesOfBGW}
Let $\mu$ be an offspring distribution 
with the generating function $f$ satisfying~\eqref{eq:critical_gen_functions} in Assumption~\ref{assumption_SL}. 
If $l_n\in\na $ is such that $z_n:=l_n/b_n\to z>0$,  where $(b_n)$ satisfies equation \eqref{equationDefbn}, 
then, as $n\to\infty$, 
\[
    \proba{d_0(\tilde Z^{z_n}_n)\leq t}
    =\proba{d_0(\tilde Z_1^{l_n})/n \leq t}
    \to e^{-\frac{z}{(\alpha c t)^{1/\alpha}}}
    =\proba{d_0(\tilde Z^z)\leq t}\quad\text{for any $t\in[0,\infty)$.}
\]
\end{lemma}
\begin{proof}
Under Assumption \ref{assumption_SL},  \cite[Lem.~2]{MR0228077} tells us that, as $n\to\infty$, we have 
\begin{equation}
    \label{eqSlackAsymptotics}
    c(1-f^{\circ n}(0))^\alpha \mathcal{l}(1-f^{\circ n}(0))\sim \frac{1}{\alpha n}.
\end{equation}
Define  $u_n:=(\alpha c)^{-1/\alpha}(1-f^{\circ n}(0))^{-1}$.
Since $\mathcal{l}$ is slowly varying at zero and $u_n\to\infty$ as $n\to\infty$, by~\eqref{eqSlackAsymptotics} 
we get $n\sim u_n^\alpha /\mathcal{l}(1/u_n)$. Thus the sequence $(u_n)$ is asymptotically equivalent  to an asymptotic inverse of the sequence $( n^\alpha/\mathcal{l}(1/n))$. Since,
by the property in~\eqref{equationDefbn}, the scaling sequence $(b_n)$ is also an asymptotic inverse of the sequence  $( n^\alpha/\mathcal{l}(1/n))$, we get 
\[
    b_n\sim u_n=(\alpha c)^{-1/\alpha}[1-f^{\circ n}(0)]^{-1}\quad\text{as $n\to\infty$.}
\]Since the sequence $(b_n)$ is regularly varying with index $1/\alpha$, the limit $b_{\floor{nt}}/b_n\to t^{1/\alpha}$ holds by definition. 
By~\eqref{eq:BGW_extiction}, 
\begin{align*}
    \proba{d_0(\tilde Z_1^{l_n})/n \leq t}
    &=f^{\circ \floor{nt}}(0)^{l_n}
    =\paren{1-(1-f^{\circ \floor{nt}}(0))}^{l_n}
    \\&\sim \paren{1- \frac{1}{(\alpha ct)^{1/\alpha}b_n}}^{b_n z}
    \to e^{-\frac{z}{(\alpha c t)^{1/\alpha}}}\quad\text{as $n\to\infty$.} 
\end{align*}

The  Laplace transforms of the one-dimensional distributions of the CBI process under $\p_z$ specialize to those of $\tilde Z^z$ upon setting $d=0$ in \eqref{equationLaplaceTransformOfCBI}. 
Since $0$ is also absorbing for $\tilde Z^z$, 
we therefore recover the expression for the hitting times of the left-hand side: 
\[
    \proba{d_0(\tilde Z^z)\leq t}
    =\proba{\tilde Z^z(t)=0}=e^{-\frac{z}{(\alpha c t)^{1/\alpha}}}. \qedhere
\]
\end{proof}

The proof of Theorem~\ref{propositionConvergenceOfPairsdepsdgepsg} requires
a result,  analogous to Lemma~\ref{lemmaConvergenceOfExtinctionTimesOfBGW}, 
for the BGWI process $Z_1$ introduced in Section~\ref{sectionIntroduction} above, cf.~\eqref{eq_discreteLampertiTransformation_1}. 
Let $Z_1^l$ be a BGWI process started at $l\in\na$ with the same offspring and immigration mechanisms as $Z_1$. In particular, note that $Z_1=Z_1^0$.
We can explicitly construct $Z^l_1$ via the discrete Lamperti transformation as in \eqref{eq_discreteLampertiTransformation_1}:
\begin{equation}
    \label{eq:Def_BGWI_started>0}
    Z^l_1=l+X_1\circ C^l_1+Y_1, 
    \quad\text{where}\quad
    C^l_1(k)=\sum_{0\leq j<k} Z^l_1(j). 
\end{equation}
Let $\p_l$ stand for the law of $Z_1^l$. 
Recall that the generating function of $Z_1^l(n)$, for any $n,l\in\na$,
can  be expressed in terms of $g$ and the iterates of $f$  
by~\eqref{eq:Def_BGWI_started>0} as follows:
\begin{equation}
    \label{eq:BDWI_product_formula}
    \imf{\mathbb{E}}{s^{Z^l_1(n)}}
    =f^{\circ n}(s)^l\cdot \prod_{0\leq m<n} g\circ f^{\circ m}(s).
\end{equation}  By~\eqref{eq:BDWI_product_formula}, we deduce the following branching property relating BGWI and BGW processes:
\begin{equation}
\label{eq:Branching_Property_BGWI}
    \p_l=\tilde \p_l*\p_0. 
\end{equation}
Differently put, 
if $Z_1$ is independent of $\tilde Z^l$, 
we have $Z_1^l\stackrel{d}{=} \tilde Z^l_1+Z_1$. 
As a consequence, we will make use of the relation of the hitting times of zero of $Z^l_1$ and $Z_1$.

\begin{lemma}
\label{lemmaPreLimitAsymptotics}
Let the offspring and immigration generating functions $f$ and $g$, respectively, satisfy Assumption \ref{assumption_SL} and denote $\delta=\frac{d}{c\alpha}$.
Then there exists a function $\mathcal{l}^*:[0,\infty)\to (0,\infty)$, slowly varying at infinity,  such that 
\begin{align}  \label{equationLocalLimitTheoremBGWIAt0}
    \proba{Z_1(n)=0}
    &\sim \mathcal{l}^*(n)n^{-\delta}\text{ as }n\to\infty.
\end{align}    
If $\delta\in(0,1)$, then we have 
\begin{align}    
\label{equationExcursionLengthAsymptoticsBGWI}
    \proba{d_1(Z_1)>n}
    &\sim\frac{1}{\Gamma(\delta)\Gamma(1-\delta) \mathcal{l}^*(n)n^{1-\delta}} = \frac{\sin(\pi\delta)}{\pi} n^{\delta-1}/\mathcal{l}^*(n)\text{ as }n\to\infty. 
\end{align}
\end{lemma}

Note that~\eqref{equationLocalLimitTheoremBGWIAt0}--\eqref{equationExcursionLengthAsymptoticsBGWI} in Lemma~\ref{lemmaPreLimitAsymptotics}  imply the asymptotic equivalence
\[
    n\proba{Z_1(n)=0}\sim \frac{1}{\Gamma(\delta)\Gamma(1-\delta) \proba{d_1(Z_1)>n}},\quad\text{as $n\to\infty$.}
\]
As  $Z_n\stackrel{d}{=}Z_1(\floor{n\cdot})/b_n$
and  $d_1(Z_1)/n\stackrel{d}{=}d_{1/n}(Z_n)$,
the scaling sequences $(c_n)$ and $(\tilde c_n)$ in Theorems~\ref{theorem_Local_time_conv} and~\ref{theorem_MUB22}, respectively, are asymptotically equivalent (up to a constant factor).

\begin{proof}[Proof of Lemma~\ref{lemmaPreLimitAsymptotics}]
We first construct the regularly varying function $\mathcal{l}^*$ and prove \eqref{equationLocalLimitTheoremBGWIAt0}. 
To this end, we start with the asymptotic analysis of the product formula~\eqref{eq:BDWI_product_formula} for $l=0$:
\[
    \proba{Z_1(n)=0}=\prod_{0\leq i<n} g\circ f^{\circ i}(0).
\]
By Assumption~\ref{assumption_SL}, $\mathcal{k}(s)/\mathcal{l}(s)\to1$ as $s\to0$ and the offspring distribution $\mu$ is critical, implying $f^{\circ n}(0)\to1$ as $n\to\infty$. The asymptotic equivalence in~\eqref{eqSlackAsymptotics} (i.e.~\cite[Lem.~2]{MR0228077}) implies
\begin{equation*}
    1-g\circ f^{\circ n}(0)
    \sim d(1-f^{\circ n}(0))^\alpha \mathcal{k}(1-f^{\circ n}(0))
    \sim \frac{d\mathcal{k}(1-f^{\circ n}(0))}{\alpha c n \mathcal{l}(1-f^{\circ n}(0))}
    \sim \frac{\delta }{n}\quad\text{as $n\to\infty$.} 
\end{equation*}
Put differently, we have $1-g\circ f^{\circ n}(0)=(1+\eps(n))\delta/n$ where $\eps(n)\to 0$ as $n\to\infty$. 
Recall that for every $s\in(0,1)$ there exists $\xi_s\in(0,s)$ such that
$\log(1-s)=-s+s^2/(2(1-\xi_s)^2)$.
Then
\begin{align*}
    \proba{Z_1(n)=0}
    &=e^{\sum_{0\leq i<n} \log(1-(1-g\circ f^{\circ i}(0)) )}
    =e^{-\sum_{0< i<n}\delta/i}\cdot
    \mathcal{l}^*(n),
\end{align*}
where  
$\mathcal{l}^*(n)\coloneqq g(0) 
    e^{-\sum_{0< i<n}\delta^2(1+\eps(i))^2/(2(1-\xi_i)^2i^2)}\cdot e^{-\sum_{0< i<n}\imf{\eps}{i}\delta/i}$
and 
$0<\xi_i<1-g\circ f^{\circ i}(0)<1$ for all $i\in\na$ (recall that  $g(0)=\nu(0)\in(0,1)$ by our standing assumption; if we had $\nu(0)=0$,  the $\bgwi$ process $Z_1$ could not return to zero, cf.~Appendix~\ref{app:example_non_conv_local_time} below).
Since $1-g\circ f^{\circ i}(0)\to0$ as $i\to\infty$, the factor $g(0) 
    e^{-\sum_{0< i<n}\delta^2(1+\eps(i))^2/(2(1-\xi_i)^2i^2)}$ in the definition of $\mathcal{l}^*$ converges to a finite positive value, 
by Karamata's representation~\cite[Thm~1.3.1]{MR898871},
$\mathcal{l}^*$ is a slowly varying function.
Note that, modulo a constant, the  factor $e^{-\sum_{0< i<n}\delta/i}$ in the expression for $\proba{Z_1(n)=0}$ above is asymptotic to $n^{-\delta}$, implying \eqref{equationLocalLimitTheoremBGWIAt0}.

The next step is to prove \eqref{equationExcursionLengthAsymptoticsBGWI}, assuming $0<\delta<1$.
We now consider the tail asymptotics for the first return to zero of $Z_1$ (which equals $d_1(Z_1)$). 
Let $U(x)=\sum_{0\leq n\leq x} \proba{Z_1(n)=0}$ be the renewal function corresponding to the returns to zero of $Z_1$. 
The (direct half of the) Karamata integral theorem (cf. \cite[Thm~1.5.11]{MR898871}) and \eqref{equationLocalLimitTheoremBGWIAt0} imply
\begin{equation}
\label{eq:asymptotics_of_U}
    U(x)\sim \frac{x^{1-\delta}}{1-\delta}\mathcal{l}^*(x)\text{ as  }x\to \infty.
\end{equation}
Let $\hat U(\lambda)\coloneqq \int_{[0,\infty)} e^{-\lambda x}\, U(dx)$ be the Laplace transform of the measure associated to $U$. Since $\{Z_1(n)=0\}=\bigcup_{k\in\na}\{D_k=n\}$,
where $D_0:=0$ and $D_k:=d_{1+D_{k-1}}(Z_1)$, $k\geq 1$, and, by the strong Markov property of $Z_1$,  we have
$\mathbb{E}(e^{-\lambda D_k})=\big(\mathbb{E}(e^{-\lambda d_1(Z_1)})\big)^{k}$. It holds that
\[
    \hat U(\lambda)
    =\sum_{n\in\na} e^{-\lambda n}\proba{Z_1(n)=0}
    =\sum_{n\in\na} e^{-\lambda n}\sum_{k\in\na}\proba{D_k=n}=
    \frac{1}{1-\mathbb{E}\paren{e^{-\lambda d_1(Z_1)}}}.
\]
By~\eqref{eq:asymptotics_of_U}, the Karamata Tauberian theorem implies
\[
    \hat U(\lambda) \sim \Gamma(2-\delta)\frac{\lambda^{-(1-\delta)}}{1-\delta}\mathcal{l}^*(1/\lambda) 
    \quad\text{ as  }\lambda\to 0,
\]so that
\begin{equation}
\label{eq:final_laplace_assym_equiv}
    1-\mathbb{E}\paren{e^{-\lambda d_1(Z_1)}}\sim \frac{\lambda^{1-\delta}}{\Gamma(1-\delta)\mathcal{l}^*(1/\lambda)}\quad\text{ as  }\lambda\to 0.  
\end{equation}
By~\eqref{eq:final_laplace_assym_equiv} and Example~c in \cite[XIII\S 5, p.447]{MR0270403} we get the asymptotic equivalence in \eqref{equationExcursionLengthAsymptoticsBGWI}. 
\end{proof}


\begin{corollary}
\label{corPreLimitAsymptotics}
Let Assumption \ref{assumption_SL} hold with $\delta=
\frac{d}{\alpha c}\in(0,1)$ and assume $Z_n\stackrel{d}{=}Z_1(\floor{n\cdot})/b_n$.
The following local limit theorem holds: for any $t>0$ and  any sequence $(s_n)$ such that $ns_n\in\na$ and $s_n\to s\in (0,t)$, we have
\begin{equation}
    \label{equationArcsineLocalLimitLaw}
    \proba{g_t(Z_n)=s_n}
    \sim \frac{1}{n}\frac{1}{\Gamma(\delta)\Gamma(1-\delta) s^\delta(t-s)^{1-\delta}}\text{ as }n\to\infty. 
\end{equation}
\end{corollary}

\begin{proof}
 The asymptotic equivalences in Lemma~\ref{lemmaPreLimitAsymptotics}
imply the local limit theorem for $g_{t}(Z_n)$ as stated in \eqref{equationArcsineLocalLimitLaw}. 
Indeed, for any sequence $(s_n)$, such that $s_n\in\na/n$ for all $n\in\na$ and $s_n\to s\in (0,t)$,  we 
have $ns_n\geq1$ for all sufficiently large $n$ and
$$\{g_t(Z_n)=s_n\}=\{g_{nt}(Z_1)/n=s_n\}=\{Z_1(ns_n-1)=0, d_{ns_n}(Z_1)>n(t-s_n)\}.$$ 
(We assume the almost sure equality $Z_n=Z_1(\floor{n\cdot})/b_n$ holds.)
By the Markov property at time $ns_n-1$, we
obtain
\begin{align*}
    \proba{g_t(Z_n)=s_n}
   &=\proba{Z_1(ns_n-1)=0}\proba{d_1(Z_1)>n(t-s_n)}
    \\
    &\sim\frac{\mathcal{l}^*(ns_n-1)(ns_n-1)^{-\delta}}{\Gamma(\delta)\Gamma(1-\delta) \mathcal{l}^*(n(t-s_n))(n(t-s_n))^{1-\delta}} \\
    &\sim \frac{1}{n}\frac{1}{\Gamma(\delta)\Gamma(1-\delta) (s_n-1/n)^\delta(t-s_n)^{1-\delta}}\\
    &\sim \frac{1}{n}\frac{1}{\Gamma(\delta)\Gamma(1-\delta) s^\delta(t-s)^{1-\delta}}\qquad\text{as $n\to\infty$,}
\end{align*}
where the first asymptotic equivalence follows from~\eqref{equationLocalLimitTheoremBGWIAt0}--\eqref{equationExcursionLengthAsymptoticsBGWI} and the second holds since $\mathcal{l}^*$ is slowly varying.
\end{proof}

The local limit theorem in Corollary~\ref{corPreLimitAsymptotics}, Scheff\'e's Lemma and a general observation about ends of excursion straddling a given time yield the following weak limit result. 

\begin{lemma}
\label{propHittingTimeConvergence}
Recall that $Z_n\stackrel{d}{=}Z_1(\floor{n\cdot})/b_n$.
Under Assumption \ref{assumption_SL} with $\delta=\frac{d}{\alpha c}\in(0,1)$, if $t_n\to t>0$, we have the weak convergence
\begin{equation}
\label{eq:weak_almost_final_limit}
    (g_{t_n}(Z_n),d_{t_n}(Z_n))\stackrel{d}{\to} (g_t(Z),d_t(Z)) \text{ as }n\to\infty. 
\end{equation}
\end{lemma}

\begin{proof}
Theorem 3.3 in \cite[Ch.1\S 3, p. 30]{MR1700749} (essentially Scheff\'e's Lemma) implies that the local limit theorem in Corollary~\ref{corPreLimitAsymptotics} above yields  
a weak limit theorem for $g_{t}(Z_n)$. By~\eqref{equationArcsineLocalLimitLaw}, the density  of the limit equals $s\mapsto 1/(\Gamma(\delta)\Gamma(1-\delta)s^{\delta}(t-s)^{1-\delta})$
on the interval $s\in(0,t)$.
This weak limit is the generalised arcsine law of parameter $1-\delta$,
which corresponds to the Beta law  $\Beta(1-\delta,\delta)$. 
Proposition 13 in \cite{MR3263091} implies that the zero set of $Z$ is the range of a stable subordinator of index $1-\delta$. Therefore, we see that $g_t(Z)/t$ also has the genralised arcsine law of parameter $1-\delta$ (as shown, for example, in Proposition 3.1 of \cite{MR1746300}) and therefore
\begin{equation}
\label{eq:weak_limit_g_t}
 g_t(Z_n)\stackrel{d}{\to} g_t(Z)\text{ as }n\to\infty,\text{ for every $t\in(0,\infty)$.}
\end{equation}

Pick any $s_1<t<s_2$ straddling $t$. Since $t_n\to t$, for all large $n$ we have  $s_1<t_n<s_2$ and 
\[
    \set{g_{t_n}(Z_n)<s_1<s_2<d_{t_n}(Z_n)}=\set{g_{s_2}(Z_n)<s_1}.
\]
Since $\set{g_{t}(Z)<s_1<s_2<d_{t}(Z)}=\set{g_{s_2}(Z)<s_1}$ and, by~\eqref{eq:weak_limit_g_t}, the following limit holds
$\proba{g_{s_2}(Z_n)<s_1}\to\proba{g_{s_2}(Z)<s_1}$ 
as $n\to\infty$ for all $s_1<s_2$, since for every $s_2$ 
the limit law in~\eqref{eq:weak_limit_g_t}
has no atoms.
By the Portmanteau Theorem, \cite[Ch.1, \S 2, Thm~2.1]{MR1700749},
we obtain the weak limit in~\eqref{eq:weak_almost_final_limit}.
In particular, note that we then also get the weak convergence of the second coordinate $d_{t_n}(Z_n)$.
\end{proof}

We establish, using the convergence of the hitting times of zero of BGWs in Lemma~\ref{lemmaConvergenceOfExtinctionTimesOfBGW},
the limit in Lemma~\ref{propHittingTimeConvergence} above and the brancing properties of both the limit and the pre-limit processes, a limit theorem for the hitting times of zero of the scaled BGWI processes started from an arbitrary positive state. 
Recall from~\eqref{eq:Def_BGWI_started>0}
the notation $Z_1^l$ for  the BGWI process started at level $l\in\na\setminus\{0\}$ and denote by $Z^z$ the CBI process started from $z>0$.  

\begin{lemma}
\label{propHittingTimeConvergence_varying_Staring_points}
Under Assumption \ref{assumption_SL} with $\delta=\frac{d}{\alpha c}\in(0,1)$,  if the sequence $(z_n)$ is such that $b_n z_n\in\na$ and $z_n\to z\in(0,\infty)$, where the scaling sequence $(b_n)$ satisfies \eqref{equationDefbn}, and $Z_n^{z_n}\stackrel{d}{=}Z_1^{b_nz_n}(\floor{n\cdot})/b_n$ we have
\begin{equation}
\label{eq:conv_moving_starting_point_d_0}
    d_0(Z_n^{z_n})\stackrel{d}{\to} d_0(Z^z) \text{ as }n\to\infty.
\end{equation}
\end{lemma}

\begin{proof}
The  branching property of BGWI processes in~\eqref{eq:Branching_Property_BGWI} implies that, starting from $l\in\na$, the hitting time of $0$ of $Z_1^{l}$ equals the first zero of $Z_1$ after $\tilde Z^l_1$  hits $0$ at $d_0(\tilde Z_1^l)$. 
Succinctly put, 
\begin{equation*}
    d_0(Z_1^l)\stackrel{d}{=}d_{d_0(\tilde Z_1^l)}(Z_1), \quad\text{where $Z_1$ and $\tilde Z_1^l$ are independent.}
\end{equation*}
This equality in law  also holds for $Z_n^{z_n}$ (resp. $Z^z$) with 
$\tilde Z_n^{z_n}$ and $Z_n$ (resp. $\tilde Z^z$ and $Z$) independent:
\begin{equation}
    \label{equationHittingTimesOfGWandGWI}
    d_0(Z_n^{z_n})\stackrel{d}{=}d_{d_0(\tilde Z_n^{z_n})}(Z_n)\quad\text{(resp.
    $d_0(Z^z)\stackrel{d}{=}d_{d_0(\tilde Z^z)}(Z)$).}
\end{equation}
For a bounded continuous $h:[0,\infty)\to\re$, define  bounded functions $H_n, H:[0,\infty)\to \re$ by
\[
    H_n(t):=\mathbb{E}(h(d_t(Z_n))),\qquad H(t):= \mathbb{E}(h(d_t(Z))).
\]
Since, by~\eqref{eq:weak_almost_final_limit} in Lemma~\ref{propHittingTimeConvergence},
it holds that 
$d_{t_n}(Z_n)\stackrel{d}{\to }d_{t}(Z)$ as $n\to\infty$ for any sequence $t_n\to t$, we have 
\begin{equation}
\label{eq:lim_H_n_to_H}
\text{$H_n(t_n)\to H(t)$ as $n\to\infty$. }
\end{equation}
As $z_n\to z$, the weak convergence 
$d_0(\tilde Z^{z_n}_n)\stackrel{d}{\to} d_0(\tilde Z^z)$ holds by Lemma~\ref{lemmaConvergenceOfExtinctionTimesOfBGW}.  
 By the Skorokhod representation theorem, we may assume almost sure convergence $d_0(\tilde Z^{z_n}_n)\to d_0(\tilde Z^z)$, which, together  with~\eqref{eq:lim_H_n_to_H}, yields  an almost sure limit $H_n(d_0(\tilde Z^{z_n}_n))\to H(d_0(\tilde Z^z))$. By~\eqref{equationHittingTimesOfGWandGWI}, we obtain 
\begin{align*}
    \mathbb{E}\paren{h(d_0(Z_n^{z_n}))}
    =
    \mathbb{E}\paren{H_n(d_0(\tilde Z^{z_n}_n))}
    \to \mathbb{E}\paren{H(d_0(\tilde Z^z))}
=\mathbb{E}\paren{h(d_0(Z^z))}\quad\text{as $n\to\infty$,} 
\end{align*}
where
the limit follows by the Dominated Convergence Theorem applied to the bounded sequence of  almost surely convergent random variables $(H_n(d_0(\tilde Z^{z_n}_n)))_{n\geq1}$.
\end{proof}

\begin{remark}
Note that the weak limit in~\eqref{eq:weak_almost_final_limit} 
appears similar to the assumptions of Theorem~\ref{theorem_MUB22} (i.e.~the convergence of $g_t(Z_n)\to g_t(Z)$ and $d_t(Z_n)\to d_t(Z)$, together with $Z_n\to Z$, in probability), which we are striving to establish, in order to apply Theorem~\ref{theorem_MUB22}. However, note also that the joint weak convergence in~\eqref{eq:weak_almost_final_limit}  of Lemma~\ref{propHittingTimeConvergence}
does not imply the requisite convergence in probability by, say, the Shorokhod representation theorem, as we have not established the weak convergence of the triplet
$(g_{t_n}(Z_n),d_{t_n}(Z_n),Z_n)\stackrel{d}{\to} (g_t(Z),d_t(Z),Z)$. 
Note also that extending directly the weak convergence in Lemma~\ref{propHittingTimeConvergence} 
to that of the triplet appears difficult.
\end{remark}

We now give a proof of Theorem~\ref{lemmaYaglomLimit}, which relies in an essential way on the asymptotics in Lemma~\ref{lemmaPreLimitAsymptotics} above.

\begin{proof}[Proof of Theorem~\ref{lemmaYaglomLimit}]
Set $\rho=d_1(Z_1)$ as the first return to zero of $Z_1$. Note that by the definition in~\eqref{eq_thmLocalTimeLimitTheoremFromMUB22_g_t_d_t} above, on the event $Z_1(1)=0$ we have $\rho=1$.  Define
$N(n,s)\coloneqq \esp{1-s^{Z_1(n\wedge \rho)}}$ and note the following: $N(0,s)=0$, the right-limit as $s\downarrow0$ equals $N(n,0)\coloneqq \lim_{s\downarrow0}N(n,s)=\proba{\rho>n}=:q_n$, 
and  $1-s^{Z_1(n\wedge \rho)}=0$ on the set $\set{\rho\leq n}=\set{Z_1(n\wedge \rho)=0}$. 
Moreover, $Z_1^k(1\wedge \rho)=Z_1^k(1)$ 
for any starting point  $k\in\na$ and,
by~\eqref{eq:BDWI_product_formula}, 
we have
$$\esp{s^{Z_1^k(1\wedge \rho)}}=\esp{s^{Z_1^k(1)}}=f(s)^k\cdot g(s).$$
Hence, by conditioning on $Z_1(n\wedge \rho)$ and the strong Markov property at $n\wedge \rho$, 
we obtain
\begin{align*}
N(n+1,s)&=\esp{1-s^{Z_1((n+1)\wedge \rho)}}
=\esp{(1-s^{Z_1((n+1)\wedge \rho)})\indi{\set{Z_1(n\wedge \rho)>0}}}\\
& = \esp{\indi{\set{Z_1(n\wedge \rho)>0}}(1-f(s)^{Z_1(n\wedge \rho)}\cdot g(s))}\\
& = \esp{\indi{\set{\rho>n}}(1-g(s)+g(s)(1-f(s)^{Z_1(n\wedge \rho)})},
\end{align*}
where $\indi{A}$ denotes the indicator of the event $A$, implying the recursion 
\[
    N(n+1,s)=\bra{1-g(s)}q_n+g(s)N(n,f(s)). 
\]
It is easily seen that by induction this recursion yields
\[
    N(n+1,s)=\sum_{k=0}^n q_{n-k}\bra{1-g\circ f^{\circ k}(s)}\Pi_k(s),
\]where by~\eqref{eq:BDWI_product_formula} we have 
\begin{equation}
\label{eq:pi_k}
\Pi_k(s)\coloneqq \prod_{0\leq i<k} g\circ f^{\circ i}(s)=\esp{s^{Z_1(k)}}. 
\end{equation}
We can rewrite the above as
\begin{equation}
    \label{equationMeanderLaplaceTransorm}
    \espc{1-e^{-\lambda Z_1((n+1)\, \wedge \rho)/b_n}}{\rho>n+1}
    =\sum_{k=0}^n \frac{q_{n-k}}{q_{n+1}}\bra{1-g\circ f^{\circ k}(e^{-\lambda/b_n})}\Pi_k(e^{-\lambda/b_n}). 
\end{equation}

We now analyse the asymptotic behaviour of each of the three factors in the summands on the right-hand side of~\eqref{equationMeanderLaplaceTransorm}.
With this in mind, pick arbitrary $\gamma\in(0,1/2)$. From the asymptotics in~\eqref{equationExcursionLengthAsymptoticsBGWI} of Lemma~\ref{lemmaPreLimitAsymptotics} above and 
the uniform convergence theorem (UCT) for slowly varying functions~\cite[Ch.1\S2,~Thm~1.2.1]{MR898871},
we obtain
\begin{align}
\frac{q_{n-k}}{q_{n+1}} & \sim \frac{\mathcal{l}^*(n(1+1/n))}{\mathcal{l}^*(n(1-k/n))} \cdot
\left(\frac{1+1/n}{1-k/n}\right)^{1-\delta}\sim (1-k/n)^{-(1-\delta)}\qquad\text{ as $n\to\infty$}
\label{eq:prob_quotient}
\end{align}
uniformly in $k\in [0, (1-\gamma) n]$ (recall that $\mathcal{l}^*$ is slowly varying at infinity). 

Since $b_n^{-1}Z_1(n)\stackrel{d}{\to} Z(1)$ as $n\to\infty$, 
the corresponding Laplace transforms converge pointwise and
the sequence of laws of $b_n^{-1}Z_1(n)$ is tight. By~\eqref{equationLaplaceTransformOfCBI} we thus get
\begin{equation}
\label{eq:loc_unif_Laplace}\Pi_n(e^{-\lambda/b_n})\to (1+\alpha c  \lambda^\alpha)^{-\delta}\quad\text{ as $n\to\infty$,}
\end{equation}
locally uniformly in $\lambda\in[0,\infty)$.
Applying the UCT~\cite[Ch.1\S2,~Thm~1.2.1]{MR898871} to the regularly varying saclling sequence $(b_n)$, satisfying~\eqref{equationDefbn}, yields 
 $b_k/b_n\sim (k/n)^{1/\alpha}$ as $n\to\infty$ uniformly for $k\in [\gamma n,(1-\gamma)n]$. 
Since convergence in~\eqref{eq:loc_unif_Laplace} is locally uniform in $\lambda$, we get 
\begin{align}
 \nonumber   \Pi_k(e^{-\lambda/b_n})&=\Pi_k(e^{-\lambda(b_k/b_n)/b_k})\sim  (1+\alpha c (b_k/b_n)^\alpha \lambda^\alpha)^{-\delta}\\
 & \sim  (1+\alpha c (k/n) \lambda^\alpha)^{-\delta}\quad\text{as $n\to\infty$,}
\label{eq:pi_k_asympt}
\end{align}
uniformly in $k\in [\gamma n,(1-\gamma)n]$.

We now analyse the asymptotic behaviour of the term $1-g\circ f^{\circ k}(e^{-\lambda/b_n})$
in the summands on the right-hand side of~\eqref{equationMeanderLaplaceTransorm}.
Recall the weak convergence of the BGW processes, started at $b_n$, to the CB process started at one: 
by~\cite[Thm~2.1]{MR2225068}, the CB process $\tilde Z^1$ is the (large population) scaling limit of $\tilde Z^1_n$, where $\tilde Z^1_n(t)=\tilde Z^{b_n}_1(\floor{nt})/b_n$, the scaling sequence $(b_n)$ satisfies the asymptotic relation~\eqref{equationDefbn} and is assumed (with out loss of generality) to be in $\na$. Thus, by~\eqref{equationLaplaceTransformOfCBI} (with $d=\delta=0$) and~\eqref{eq:BGW_transform} above, the following limit holds
\[
    f^{\circ n}(e^{-\lambda /b_n})^{b_n}
    =\esp{e^{-\lambda \tilde Z^{b_n}_1(n)/b_n}} \to \esp{e^{-\lambda \tilde Z^1(1)}}=
    e^{-\lambda/(1+\alpha c t \lambda^\alpha)^{1/\alpha}}\quad\text{as $n\to\infty$.}
\]
Taking logarithms and applying Taylor expansion at $1$ yields
\begin{equation}
\label{eq:asympt_b_n_l}
    (1-f^{\circ k}(e^{-\lambda/b_k}))
    \sim \frac{\lambda}{b_k (1+\alpha c\lambda^\alpha)^{1/\alpha}}
    \qquad\text{as $k\to\infty$.}
\end{equation}
By the representation of the immigration moment generating function $g$ in~\eqref{eq:critical_gen_functions} of  Assumption~\ref{assumption_SL} and the asymptotic relation in~\eqref{eq:asympt_b_n_l}, we obtain
\begin{align}
\nonumber
  1-g\circ f^{\circ k}(e^{-\lambda/b_n})
    &= d(1-f^{\circ k}(e^{-(b_k/b_n)\lambda/b_k }))^{\alpha}\mathcal{k}(1-f^{\circ k}(e^{-(b_k/b_n)\lambda/b_k}))\\
\nonumber
    & \sim d \frac{\lambda^\alpha}{b_n^\alpha(1+\alpha c(\lambda b_k/b_n)^\alpha)}\mathcal{k}\left(\frac{1}{b_n}\frac{ \lambda}{(1+\alpha c(\lambda b_k/b_n)^\alpha)^{1/\alpha}} \right)
    \\
\nonumber    
    &\sim \frac{d}{n\mathcal{l}(1/b_n)}\frac{\lambda^\alpha }{1+\alpha c (k/n) \lambda^\alpha}\mathcal{k}(1/b_n)\\
    &
    \sim \frac{d}{n}\frac{\lambda^\alpha}{1+\alpha c (k/n) \lambda^\alpha}\qquad\text{as $n\to\infty$,}
\label{eq:final_1_over_n_term}
\end{align}
uniformly in  $k\in [\gamma n,(1-\gamma)n]$.
The second asymptotic equivalence in the last display follows from the  property 
$b_n^\alpha \sim n\mathcal{l}(1/b_n)$
in~\eqref{equationDefbn}  and the UCT~\cite[Ch.1\S2,~Thm~1.2.1]{MR898871} applied to the slowly varying function $\mathcal{k}$, while 
the third equivalence holds by the asymptotic equivalence of $\mathcal{k}$ and $\mathcal{l}$ in Assumption~\ref{assumption_SL}.

The asymptotic equivalences in~\eqref{eq:prob_quotient},~\eqref{eq:pi_k_asympt} and~\eqref{eq:final_1_over_n_term} imply that for every $\eps>0$ there exists $n_\eps\in\na$
such that for all $n\geq n_\eps$ the following inequalities hold, uniformly in $k\in [\gamma n,(1-\gamma)n]$:
\begin{align*}
    (1-\eps)
\frac{d}{n}\frac{\lambda^\alpha}{1+\alpha c (k/n) \lambda^\alpha}& (1-k/n)^{-(1-\delta)}(1+\alpha c (k/n) \lambda^\alpha)^{-\delta}\\
& \leq \frac{q_{n-k}}{q_{n+1}}\bra{1-g\circ f^{\circ k}(s)}\Pi_k(s) \\
& \leq (1+\eps)
\frac{d}{n}\frac{\lambda^\alpha}{1+\alpha c (k/n) \lambda^\alpha} (1-k/n)^{-(1-\delta)}(1+\alpha c (k/n) \lambda^\alpha)^{-\delta}.
\end{align*}
Summing the upper bounds in this display over the integers $k\in\{\lceil \gamma n\rceil,\ldots,\lfloor (1-\gamma)n\rfloor\}$ and taking the limit as $n\to\infty$ yields 
$(1+\eps)d\int_\gamma^{1-\gamma} (1-t)^{-(1-\delta)}\frac{1}{(1+\alpha c t \lambda^\alpha)^\delta} 
    \frac{\lambda^\alpha}{1+\alpha c t \lambda^\alpha}  \, dt$ (the Riemann sums converge to the integral).
The lower bound has the same limit with the factor $(1-\eps)$ instead of $(1+\eps)$.
Since $\eps>0$ was arbitrary, we obtain
\begin{equation}
	\lim_{n\to\infty}\sum_{k=\lceil\gamma n\rceil}^{\lfloor(1-\gamma)n\rfloor} \frac{q_{n-k}}{q_{n+1}}\bra{1-g\circ f^{\circ k}(s)}\Pi_k(s)
	=d\int_\gamma^{1-\gamma} (1-t)^{-(1-\delta)}\frac{1}{(1+\alpha c t \lambda^\alpha)^\delta} 
    \frac{\lambda^\alpha}{1+\alpha c t \lambda^\alpha}  \, dt.
\label{eq:most_of_integral}
\end{equation}

The limit of the right-hand side in~\eqref{equationMeanderLaplaceTransorm}
requires the analysis of terms corresponding to $k\in[0,\gamma n]\cup[(1-\gamma)n,n]$. Recall $f(s)=\sum_{n\in\na}s^n\mu(n)$, $s\in[0,1)$, where $\mu$ is a critical (i.e. $\sum_{n\in\na} n\mu(n)=1$) offspring distribution on $\na$, satisfying $\mu(1)<1$. Since $f(0)=\mu(0)<1$, $f(1)=\mu(\na)=1$, $f'(0)=\mu(1)\in(0,1)$,
$f'(s)\uparrow 1$ as $s\uparrow1$ and $f''>0$ on $(0,1)$ (making $f$ convex), the sequence 
$(f^{\circ k}(s))_{k\in\na}$ is increases  for every $s\in(0,1)$. Since the immigration generating function $g$ is increasing, 
we have the following uniform bounds for all $k\in\na$:
\begin{align}
\nonumber
	0\leq 1-g\circ f^{\circ k}(e^{-\lambda/b_n})
	&\leq 1-g(e^{-\lambda/b_n})
	\sim d(1-e^{-\lambda/b_n})^\alpha \mathcal{k}(1-e^{-\lambda/b_n})\\
    \nonumber
	& \sim d\lambda^\alpha \frac{\mathcal{k}(1/b_n)}{b_n^\alpha}
    \sim d\lambda^\alpha \frac{\mathcal{k}(1/b_n)}{\mathcal{l}(1/b_n)}\frac{\mathcal{l}(1/b_n)}{b_n^\alpha}\\
    &
	\sim \frac{d\lambda^\alpha}{n}\qquad \text{as }n\to\infty.
    \label{eq:asym_bound}
\end{align}
The last equivalence follows from the asymptotic equivalence of the slowly varying functions $\mathcal{k}$ and $\mathcal{l}$ in Assumption~\ref{assumption_SL} and the property  $b_n^\alpha \sim n\mathcal{l}(1/b_n)$
in~\eqref{equationDefbn}.

Consider $k\in[0,\gamma n]$. By~\eqref{eq:pi_k} we have $\Pi_k(s)<1$ for all $s\in(0,1)$ and $k\in\na$.
By UCT~\cite[Ch.1\S2,~Thm~1.2.1]{MR898871} applied to the slowly varying function $\mathcal{l}^*$ in~\eqref{eq:prob_quotient}, the limit  $\lim_{n\to\infty}\sup_{k\leq \gamma n}\mathcal{l}^*(n+1)/\mathcal{l}^*(n(1-k/n))=1$ holds. Together with~\eqref{eq:asym_bound}, this implies
\begin{align*}
	0&\leq \limsup_{n\to\infty}\sum_{k=0}^{\lfloor \gamma n\rfloor} \frac{q_{n-k}}{q_{n+1}}\bra{1-g\circ f^{\circ k}(s)}\Pi_k(s)
	\\ 
    &\leq  \limsup_{n\to\infty} \bra{\sup_{k\leq \gamma n}\frac{\mathcal{l}^*(n+1)}{\mathcal{l}^*(n-k)}} \frac{d\lambda^\alpha}{n}  \sum_{k=0}^{\lfloor\gamma n\rfloor} (1-\frac{k}{n})^{\delta-1}
	= d\lambda^\alpha (1-(1-\gamma)^\delta)/\delta
\end{align*}
Since $\gamma\in(0,1/2)$ was arbitrary, we obtain
\begin{equation}
\label{eq:asym_bound_1}
   0\leq \limsup_{\gamma\to 0}\limsup_{n\to\infty}\sum_{k=0}^{\lfloor \gamma n\rfloor} \frac{q_{n-k}}{q_{n+1}}\bra{1-g\circ f^{\circ k}(s)}\Pi_k(s)
   \leq \limsup_{\gamma\to 0}d\lambda^\alpha (1-(1-\gamma)^\delta)/\delta
	=0. 
\end{equation}

Consider now $k\in[(1-\gamma)n,n]$. 
An integral of a regularly varying function $m\mapsto m^{\delta-1}\mathcal{l}^*(m)$ is again regularly varying with the same slowly varying function $\mathcal{l}^*$
and index $\delta$, see~\cite[Prop.~1.5.8]{MR898871}. 
Thus
by the asymptotic equivalence~\eqref{equationExcursionLengthAsymptoticsBGWI} of Lemma~\ref{lemmaPreLimitAsymptotics} above, giving the asymptotic behaviour of the tail probability $q_k$,
we obtain
\begin{equation}
\label{eq:int_slow_vary}
\delta \Gamma(\delta)\Gamma(1-\delta) \sum_{k=0}^{\floor{\gamma n}}
q_k
\sim
\delta \sum_{k=0}^{\floor{\gamma n}}k^{\delta-1}/\mathcal{l}^*(k)\sim 
 (\gamma n)^\delta/ \mathcal{l}^*(\gamma n)\sim 
\gamma^\delta n^\delta /\mathcal{l}^*(n)\quad\text{as $n\to\infty$.}    
\end{equation}
Recalling $\Pi_k(s)<1$ (for all $s\in(0,1)$ and $k\in\na$) and the asymptotic bound in~\eqref{eq:asym_bound}, the asymptotic equivalence in~\eqref{eq:int_slow_vary}
implies
\begin{align*}
	0 &\leq \limsup_{n\to\infty}\sum_{k=\ceil{(1-\gamma)n}}^{n} \frac{q_{n-k}}{q_{n+1}}\bra{1-g\circ f^{\circ k}(s)}\Pi_k(s)
    \leq\limsup_{n\to\infty}\frac{d\lambda^\alpha}{n} \sum_{k=\ceil{(1-\gamma)n}}^{n}\frac{q_{n-k}}{q_{n+1}}
	 \\&= \limsup_{n\to\infty}\frac{d\lambda^\alpha}{n} \frac{1}{q_{n+1}}\sum_{k=0}^{\floor{\gamma n}}q_k
     = \limsup_{n\to\infty} \frac{d\lambda^\alpha}{n} \frac{\gamma^\delta n^\delta}{\mathcal{l}^*(n)\delta \Gamma(\delta)\Gamma(1-\delta) } \frac{1}{q_{n+1}} 
	\\& = \limsup_{n\to\infty}\frac{d\lambda^\alpha\gamma^\delta }{\delta}
  \frac{n^{\delta-1}}{\mathcal{l}^*(n)} (n+1)^{1-\delta}\mathcal{l}^*(n+1).
\end{align*}
The  last equality in the previous display follows from the asymptotic equivalence in~\eqref{equationExcursionLengthAsymptoticsBGWI} of Lemma~\ref{lemmaPreLimitAsymptotics} applied to the tail probability $q_{n+1}$. As $\gamma\in(0,1/2)$ was arbitrary, we may take the following limit:
\begin{align}
\nonumber
0 &\leq \limsup_{\gamma\to 0}\limsup_{n\to\infty}\sum_{k=\ceil{(1-\gamma)n}}^{n} \frac{q_{n-k}}{q_{n+1}}\bra{1-g\circ f^{\circ k}(s)}\Pi_k(s)
\\& \leq \limsup_{\gamma\to 0}\limsup_{n\to\infty}\frac{d\lambda^\alpha\gamma^\delta }{\delta}
  \frac{n^{\delta-1}}{\mathcal{l}^*(n)} (n+1)^{1-\delta}\mathcal{l}^*(n+1) 
	=0. 
\label{eq:asym_bound_2}
\end{align}

The limits in~\eqref{eq:most_of_integral},~\eqref{eq:asym_bound_1} and~\eqref{eq:asym_bound_2}, together with the representation of the conditional expectation 
$\espc{1-e^{-\lambda Z_1((n+1)\, \wedge \rho)/b_n}}{\rho>n+1}$ in~\eqref{equationMeanderLaplaceTransorm}, imply
\begin{align}
\nonumber
    \lim_{n\to\infty }\espc{1-e^{-\lambda Z_1((n+1)\, \wedge \rho)/b_n}}{\rho>n+1}
    &=d\int_0^1 (1-t)^{-(1-\delta)}\frac{1}{(1+\alpha c t \lambda^\alpha)^\delta} 
    \frac{\lambda^\alpha}{1+\alpha c t \lambda^\alpha}  \, dt
    \\&\nonumber
   = d\lambda^\alpha\int_0^1 (1-t)^{\delta-1}(1+\alpha c t \lambda^\alpha)^{-(1+\delta)}
   \, dt\\ &
    =\frac{\alpha c \lambda^\alpha}{1+\alpha c\lambda^\alpha}. 
    \label{eq:explicit_formula}
\end{align}
The formula in~\eqref{eq:explicit_formula} follows from
Euler's integral formula for the hypergeometric series,
\[
	\int_0^1 (1-t)^{\delta-1}(1+zt)^{-(1+\delta)}\, dt=\Beta(1,\delta)\, _2F_1(1+\delta,1,1+\delta;-z),
\] 
and the facts that $_2F_1(1+\delta,1,1+\delta;-z)$ is in fact a geometric series equal to $1/(1-z)$ (cf. \cite[Def.~2.1.5 \& Thm~2.2.1]{MR1688958}), the Beta function $\Beta$ satisfies $\Beta(1,\delta)=1/\delta$ and  $\delta=d/\alpha c$.
Note that the right-hand side of the equality in~\eqref{eq:explicit_formula} depends neither on $\delta$ nor on $d$.

Since the limit in~\eqref{eq:explicit_formula} is that of the Laplace transforms of probability measures on $[0,\infty)$, the convergence in~\eqref{eq:explicit_formula} holds locally uniformly in $\lambda$. 
Since the scaling sequence $(b_n)$ is regularly varying (see~\eqref{equationDefbn} above),  the limit of the Laplace transforms in Theorem~\ref{lemmaYaglomLimit} for the sequence of times $t_n=1$ with limit $t=1$ holds:
\begin{equation}
\label{eq:final_laplace_Trans_1}
	\lim_{n\to\infty} \espc{e^{-\lambda Z_1(n)/b_n}}{\rho>n}=\lim_{n\to\infty} \espc{e^{-\lambda \frac{b_n}{b_{n+1}}Z_1(n+1)/b_n}}{\rho>n+1}=\frac{1}{1+\alpha c \lambda^\alpha}. 
\end{equation}
For an arbitrary sequence of times $(t_n)$, satisfying $t_n\to t>0$, consider an eventually monotonically increasing subsequence $\floor{nt_n}$ of integers and note that the locally uniform convergence of the Laplace transforms in~\eqref{eq:final_laplace_Trans_1}
implies
\begin{align*}
    \lim_{n\to\infty} \espc{e^{-\lambda Z_1(\floor{nt_n})/b_n}}{\rho>\floor{nt_n}}
    &=\lim_{n\to\infty} \espc{e^{-\lambda(b_{\floor{nt_n}}/b_n) Z_1(\floor{nt_n})/b_{\floor{nt_n}}}}{\rho>\floor{nt_n}}
    \\&=\lim_{n\to\infty} \espc{e^{-(\lambda t^{1/\alpha} Z_1(\floor{nt_n})/b_{\floor{nt_n}}}}{\rho>\floor{nt_n}}
    =\frac{1}{1+\alpha ct \lambda^\alpha},
\end{align*}
since 
$b_{\floor{nt_n}}/b_n\to t^{1/\alpha}$ as $n\to\infty$ by~\eqref{equationDefbn} above.

Since, assuming $Z_n=Z_1(\floor{n\cdot})/b_n$, we have $\{d_{1/n}(Z_n)>t_n\}=\{\rho=d_1(Z_1)>\floor{nt_n}\}$, the formula in~\eqref{equationYaglomLimit} follows concluding the proof of the theorem.
\end{proof}

The following proposition is crucial in the proof of Theorem~\ref{propositionConvergenceOfPairsdepsdgepsg} and hence also in the proof of our main result in Theorem~\ref{theorem_Local_time_conv}.
The proof of Proposition~\ref{lemmaCBIExcursionStraddlingt} is based on the following fact from excursion theory: the excursion of $Z$ straddling a time $t$, conditioned on its age at time $t$, has the same law as the first excursion of $Z$ exceeding that given age~\cite[Thm~7.35]{MR525052}.

\begin{proposition}
\label{lemmaCBIExcursionStraddlingt}
Under Assumption \ref{assumption_SL} with $\delta=\frac{d}{\alpha c}\in(0,1)$, for any $0<s<t$, the conditional law of $Z(t)$, given $g_t(Z)=s$, is the Linnik law $\nu^{t-s}$ on positive reals with Laplace transform $\esp{e^{-\lambda Z(t)}|g_t(Z)=s}=(1+\alpha c (t-s) \lambda^\alpha)^{-1}$.
\end{proposition}

\begin{proof}[Proof of Proposition~\ref{lemmaCBIExcursionStraddlingt}]
Let $\tau_l(Z):=\inf\set{t\geq 0: t-g_t(Z)>l}$ be the first time an excursion away from $0$ of $Z$ exceeds length $l$. 
As mentioned above, Theorem~7.35 of \cite{MR525052} shows that the conditional law of $Z(t)$, given $g_t(Z)=t-l$, equals the law of $Z(\tau_{l}(Z))$. This implies
\begin{equation}
\label{eq:last_zero_Decomp}
\esp{e^{-\lambda Z(t)}}
	=\int_0^t \proba{g_t(Z)\in ds} \esp{e^{-\lambda Z(\tau_{t-s}(Z))}}\qquad\text{for all $t\geq0$.}
\end{equation}

The process $Z$ is self-similar of index $\alpha$, so that for each $c>0$ we have that the process $\bar Z^c=(\bar Z^c(t))_{t\in[0,\infty)}$, $\bar Z^c(t):=c^{-1/\alpha}Z(ct)$, satisfies $\bar Z^c\stackrel{d}{=} Z$, implying
$g_t(\bar Z^c)=g_{ct}(Z)/c$
and
$\tau_l(\bar Z^c)=\tau_{cl}(Z)/c$. 
Hence
$Z(\tau_l(Z))
	\stackrel{d}{=} \bar Z^c(\tau_l(\bar Z^c))=c^{-1/\alpha}Z(c\tau_l(\bar Z^c))=c^{-1/\alpha}Z(\tau_{cl}(Z))$. Setting $c=1/l$  yields
\begin{equation}
\label{eq:equalities_in_law_g_t}
	Z(\tau_l(Z))
	\stackrel{d}{=} 
	l^{1/\alpha}Z(\tau_{1}(Z))
    \quad\text{for any $l>0$.}
\end{equation}
The Laplace transform of $Z(1)$ in~\eqref{equationLaplaceTransformOfCBI}, the identity in~\eqref{eq:last_zero_Decomp} (for $t=1$) 
and the equality in law in~\eqref{eq:equalities_in_law_g_t} (for $l=1-s$)
imply
\begin{align*}
(1+\alpha c \lambda^\alpha)^{-\delta}=	\esp{e^{-\lambda Z(1)}}
	&=\int_0^1 \proba{g_1(Z)\in ds} \esp{e^{-\lambda (1-s)^{1/\alpha}Z(\tau_1(Z)) }} 
=	\esp{e^{-\lambda (1-G)^{1/\alpha} R}},
\end{align*}
where $G$ and $R$ are independent variables with the laws of $g_1(Z)$ and $Z(\tau_1(Z))$, respectively.

Let $\Theta_\delta$ (resp. $\Sigma$) be a gamma (resp. 
  positive stable) random variable with Laplace transform $\esp{e^{-\lambda \Theta_\delta}}= 1/(1+\lambda)^\delta$  (resp. $\esp{e^{-\lambda \Sigma}}= e^{-\lambda^\alpha}$) for $\lambda\geq0$. Assuming $\Theta_\delta$ and $\Sigma$ are independent, we get
$\esp{e^{-\lambda (\alpha c  \Theta_\delta)^{1/\alpha} \Sigma}}=1/(1+\alpha c \lambda^\alpha)^\delta$, which implies
\begin{equation}
\label{eq:equality_law}
	(\alpha c  \Theta_\delta)^{1/\alpha}\Sigma
	\stackrel{d}{=}
	 (1-G)^{1/\alpha} R.
\end{equation}
It follows for example from the local limit theorem in Corollary~\ref{corPreLimitAsymptotics} and Lemma~\ref{propHittingTimeConvergence} that 
$1-G$ 
has the generalised arcsine law with parameter $\delta$. Equivalently, $1-G$ has a density on $(0,1)$ proportional to $s\mapsto s^{-(1-\delta)}(1-s)^{-\delta}$, yielding the representation $1-G
\stackrel{d}{=}\Theta_\delta/\Theta_1$, where 
$\Theta_1\coloneqq \Theta_\delta+\Theta_{1-\delta}$ and  
$\Theta_\delta$, $\Theta_{1-\delta}$ 
are independent gamma distributed  random variables  with Laplace transforms $1/(1+\lambda)^\delta$, $1/(1+\lambda)^{1-\delta}$, respectively. It is well known that in this case the quotient $\Theta_\delta/\Theta_1$ is independent of $\Theta_1$, and thus also of $(\alpha c  \Theta_1 )^{1/\alpha}\Sigma$.
We may assume that $R$ is also independent of $\Theta_\delta/\Theta_1$.
Then,
by~\eqref{eq:equality_law}, we obtain
\[
    (\Theta_\delta/\Theta_1)^{1/\alpha}(\alpha c  \Theta_1 )^{1/\alpha}\Sigma
    \stackrel{d}{=}
    (\Theta_\delta/\Theta_1)^{1/\alpha} R. 
\]
Since $0<\Theta_\delta/\Theta_1<1$ almost surely and 
$\esp{(\Theta_\delta )^{s/\alpha}\Sigma^s}<\infty$
for all $s\in[0,\alpha)$,
we obtain
$$
\esp{((\alpha c  \Theta_1 )^{1/\alpha}\Sigma)^s}=
\frac{\esp{((\alpha c  \Theta_\delta)^{1/\alpha}\Sigma)^s}}{\esp{(\Theta_\delta/\Theta_1)^{s/\alpha}}}
=
\frac{\esp{(\Theta_\delta/\Theta_1)^{s/\alpha}R^s}}{\esp{(\Theta_\delta/\Theta_1)^{s/\alpha}}}=\esp{R^s}\quad\text{for all $s\in [0,\alpha)$.}
$$
In particular, since the Mellin transform determines a law on $[0,\infty)$ uniquely, we get
\[
    (\alpha c  \Theta_1)^{1/\alpha}\Sigma
    \stackrel{d}{=}
    R
\]
(cf. Exercise 1.14 in \cite{MR2016344}).
Note that $\delta$ is no longer present on the left-hand side of the last  identity, implying that the law of $R$ does not depend on $\delta$. 
Since $\Theta_1$
is exponential with unit mean, its Laplace transform equals
$\esp{e^{-\lambda \Theta_1}}=1/(1+\lambda)$, implying 
\[
\esp{e^{-\lambda R}}=\esp{e^{-\lambda (\alpha c  \Theta_1)^{1/\alpha}\Sigma}}=
\esp{e^{-\lambda^\alpha \alpha c  \Theta_1}}=
1/(1+\alpha c \lambda^{\alpha})\qquad\text{for all $\lambda\geq0$.}
\]
Hence $R$ follows the Linnik law $\nu^1$.
By~\eqref{eq:equalities_in_law_g_t} for any $l>0$ we have
$Z(\tau_l(Z)) \stackrel{d}{=}l^{1/\alpha}R$.
Hence for any $0<s<t$, we obtain
$$
\esp{e^{-\lambda Z(t)}|g_t(Z)=s}=\esp{e^{-\lambda Z(\tau_{t-s}(Z))}}=\esp{e^{-\lambda (t-s)^{1/\alpha}R}}=1/(1+\alpha c (t-s)\lambda^{\alpha}),
$$
as claimed in the lemma.
\end{proof}

\subsection{Conclusion of the proof of Theorem~\ref{propositionConvergenceOfPairsdepsdgepsg}}
Recall that Theorem~\ref{propositionConvergenceOfPairsdepsdgepsg} consists of two weak limits for any fixed $t\geq0$ and $\eps>0$ as $n\to\infty$:
\begin{align}
\label{eq:second_weak_limit_in_Thm_5}
   & (d^\eps_t(Z_n), d_t(Z_n))\stackrel{d}{\to} (d^\eps_t(Z), d_t(Z)),\\
 &   (g^\eps_t(Z_n), g_t(Z_n))\stackrel{d}{\to} (g^\eps_t(Z), g_t(Z)).
 \label{eq:first_weak_limit_in_Thm_5}
\end{align}

\subsubsection{Proof of the weak limit in~\eqref{eq:second_weak_limit_in_Thm_5}}
Note that, conditionally on $d_t^\eps(Z_n)$, the increment $d_t(Z_n)-d^\eps_t(Z_n)$ has the law of the hitting time of zero of the scaled BGWI started at $Z_n(d^\eps_t(Z_n))$.
By using two continuous and bounded functions $h_1,h_2:[0,\infty)\to\re$,  defining $H_n(z):= \mathbb{E}(h_2(d_0(Z_n^{z})))$ for any $z\in\na/b_n$, and applying the strong Markov property at the stopping time $d^\eps_t(Z_n)$, we get
\begin{align}\label{eq:joint_limit}   \mathbb{E}\paren{h_1(d^\eps_t(Z_n))h_2(d_t(Z_n)-d^\eps_t(Z_n))}    =\mathbb{E}\paren{h_1(d^\eps_t(Z_n))H_n(Z_n(d_t^\eps(Z_n)))}.  
\end{align}

In the proof of the case  $\lim_{\eps\to 0}\lim_{n\to\infty}\mathrm{II}_{d}=0$ in Section~\ref{sectionProofOverview} above we showed that Assumption~\eqref{eq:Skorkhod_lem_ass1} of Lemma~\ref{lemmaSkorokhodSpaceHittingTimeOfOpenSet} holds at $T_O(Z(t+\cdot))$, where 
$Z(t+\cdot)$ is a CBI process, started at $Z(t)$, with the same semigroup as $Z$, the open set $O$ equals the interval $[0,\eps)$
and the equality
$d_t^\eps(Z)=t+T_O(Z(t+\cdot))<\infty$
holds.
In order to verify Assumption~\eqref{eq:Skorkhod_lem_ass2} of Lemma~\ref{lemmaSkorokhodSpaceHittingTimeOfOpenSet}, we need to prove that the process $Z(t+\cdot)$ is continuous at  $T_O(Z(t+\cdot))$, which is clearly equivalent to the continuity of $Z$ at $d_t^\eps(Z)$. 
On the event $\{Z(t)> \eps\}$, $d^\eps_t(Z)$ is predicted by strictly increasing stopping times: $d^{\eps+\eps'}_t(Z)\uparrow d^\eps_t(Z)$ almost surely on $\{Z(t)> \eps\}$ as $\eps'\downarrow0$. 
On the complement $\{Z(t)\leq \eps\}$,  since $Z$ is downwards regular, we have $d^\eps_t(Z)=t$.
In both cases $Z$ is continuous at $d^\eps_t(Z)$ by quasi left continuity, implying   that assumption~\eqref{eq:Skorkhod_lem_ass2} of Lemma~\ref{lemmaSkorokhodSpaceHittingTimeOfOpenSet} also holds. Thus, by Lemma~\ref{lemmaSkorokhodSpaceHittingTimeOfOpenSet} we obtain
$(d^\eps_t(Z_n),Z_n(d^\eps_t(Z_n)))\stackrel{\text{a.s.
}}{\longrightarrow}(d^\eps_t(Z),Z(d^\eps_t(Z)))$  as $n\to\infty$.

For any sequence $(z_n)$ converging to $z\in(0,\infty)$, such that $z_nb_n\in\na$, the
weak limit in~\eqref{eq:conv_moving_starting_point_d_0}  of Lemma~\ref{propHittingTimeConvergence_varying_Staring_points} yields  
\begin{equation}
\label{eq:z_n_convergence}
H_n(z_n)= \mathbb{E}(h_2(d_0(Z_n^{z_n})))\to \mathbb{E}(h_2(d_0(Z^z)))\eqqcolon H(z)\quad\text{as $n\to\infty$.}
\end{equation}
Thus $h_1(d^\eps_t(Z_n))H_n(Z_n(d_t^\eps(Z_n)))\stackrel{\text{a.s.
}}{\longrightarrow}h_1(d^\eps_t(Z))H(Z(d^\eps_t(Z)))$  as $n\to\infty$
(recall that  $h_1$ is continuous and set
$z_n=Z_n(d_t^\eps(Z_n))\to Z(d^\eps_t(Z))=z$ in~\eqref{eq:z_n_convergence}).
Since, 
conditional on $d_t^\eps(Z)$, the increment $d_t(Z)-d^\eps_t(Z)$ has the law of the hitting time of zero of the CBI process started at $Z(d^\eps_t(Z))$ and the functions 
$h_1$ and $h_2$ are bounded, as $n\to\infty$, the Dominated Convergence Theorem implies
\begin{align*}
\mathbb{E}\paren{h_1(d^\eps_t(Z_n))H_n(Z_n(d_t^\eps(Z_n)))}
&\to\mathbb{E}\paren{h_1(d^\eps_t(Z))H(Z(d^\eps_t(Z)))}
  \\&=\mathbb{E}\paren{h_1(d^\eps_t(Z))h_2(d_t(Z)-d^\eps_t(Z))}. 
\end{align*} 
This limit and~\eqref{eq:joint_limit} 
imply the weak limit in~\eqref{eq:second_weak_limit_in_Thm_5}.

\subsubsection{Proof of the weak limit in~\eqref{eq:first_weak_limit_in_Thm_5}}
In the case $t=0$, the limit in~\eqref{eq:first_weak_limit_in_Thm_5} holds by definition since $Z(0)=Z_n(0)=0$ and thus $g_0(Z)=g_t^\eps(Z)=g_0(Z_n)=g_t^\eps(Z_n)=0$ (see definitions in~\eqref{eq_thmLocalTimeLimitTheoremFromMUB22_g_t_d_t} and~\eqref{eq_thmLocalTimeLimitTheoremFromMUB22_geps_t_deps_t} above).  In the remainder of the proof of~\eqref{eq:first_weak_limit_in_Thm_5} we assume $t>0$.
 
\begin{lemma}
\label{prop:g_t,Z(t)-limit}
Recall that $Z_n\stackrel{d}{=}Z_1(\floor{n\cdot})/b_n$.
Under Assumption \ref{assumption_SL} with $\delta=\frac{d}{\alpha c}\in(0,1)$,
for $t>0$,
the following weak limit holds 
\begin{equation}
\label{eq:weak_limit_g_and_Z}
	(g_t(Z_n), Z_n(t))\stackrel{d}{\to} (g_t(Z), Z(t))\quad\text{as $n\to\infty$.}
\end{equation}
\end{lemma}

We will establish  limit~\eqref{eq:weak_limit_g_and_Z} in Lemma~\ref{prop:g_t,Z(t)-limit} via  excursion theory and the  Yaglom limit for BGWIs in Theorem~\ref{lemmaYaglomLimit}. We now prove the weak limit in~\eqref{eq:first_weak_limit_in_Thm_5} using Lemma~\ref{prop:g_t,Z(t)-limit}.

Let $z_n\to z$ in $[0,\infty)$ and assume $b_nz_n\in\na$ for all $n\in\na$. Then, by~\cite[Thm~2.1]{MR2225068}, $Z_n^{z_n}\stackrel{d}{\to} Z^z$ and we may assume (by e.g. the Skorokhod representation) that the convergence is in fact almost sure. 
As in the proof of the case  $\lim_{\eps\to 0}\lim_{n\to\infty}\mathrm{II}_{d}=0$ in Section~\ref{sectionProofOverview} above,
it follows that Assumption~\eqref{eq:Skorkhod_lem_ass1} of Lemma~\ref{lemmaSkorokhodSpaceHittingTimeOfOpenSet} holds at $d_0^\eps(Z^z)=T_O(Z^z)$ (where $O=[0,\eps)$) for the CBI process $Z^z$ started at $z\in[0,\infty)$.
By Lemma~\ref{lemmaSkorokhodSpaceHittingTimeOfOpenSet} 
we obtain almost sure convergence $d_0^\eps(Z_n^{z_n})\to d_0^\eps(Z^z)$, and thus for any continuous and bounded $h:[0,\infty)\to\re$ we have 
\begin{equation}
\label{eq:weak_limit_hitting_time_eps}
H_n(z_n)\coloneqq \mathbb{E}(h(d_0^\eps(Z_n^{z_n})))\to\mathbb{E}(h(d_0^\eps(Z^z)))\eqqcolon H(z)
\qquad \text{as $n\to\infty$.}
\end{equation}

For any $0<s_1<s_2<t$, we have $\{g_t(Z_n)<s_1, g^\eps_t(Z_n)<s_2\}=\{g_{s_2}(Z_n)<s_1, t<d_{s_2}^\eps(Z_n)\}$
and 
$\{g_t(Z)<s_1, g^\eps_t(Z)<s_2\}=\{g_{s_2}(Z)<s_1, t<d_{s_2}^\eps(Z)\}$.
Hence~\eqref{eq:first_weak_limit_in_Thm_5} holds once we prove the following weak limit for any $s\in(0,\infty)$:
\begin{equation}
\label{eq:lim_g_d_eps}
(g_s(Z_n),d_s^\eps(Z_n))\stackrel{d}{\to}(g_s(Z),d_s^\eps(Z)) \qquad\text{as $n\to\infty$.}
\end{equation}
To establish~\eqref{eq:lim_g_d_eps},  pick bounded, continuous functions $h_1,h_2:[0,\infty)\to\re$.
The weak limit in~\eqref{eq:weak_limit_g_and_Z} of Lemma~\ref{prop:g_t,Z(t)-limit} and Skorokhod's representation imply that we may assume 
almost sure 
convergence 
$(g_s(Z_n), Z_n(s))\to (g_s(Z), Z(s))$ holds.
By the Markov property at time $s$ and the limit in~\eqref{eq:weak_limit_hitting_time_eps} (for $h=h_2(\cdot+s)$),  we 
obtain the following limit as $n\to\infty$,
\begin{align*}
\mathbb{E}(h_1(g_s(Z_n))h_2(d_s^\eps(Z_n)))&=
\mathbb{E}(h_1(g_s(Z_n))\mathbb{E}_{Z_n(s)}(h_2(d_0^\eps(Z_n)+s)))= \mathbb{E}(h_1(g_s(Z_n)) H_n(Z_n(s)))\\
& \to
\mathbb{E}(h_1(g_s(Z)) H(Z(s)))=\mathbb{E}(h_1(g_s(Z))\mathbb{E}_{Z(s)}(h_2(d_0^\eps(Z)+s)))\\
&=
\mathbb{E}(h_1(g_s(Z))h_2(d_s^\eps(Z))),
\end{align*}
implying~\eqref{eq:lim_g_d_eps} and thus~\eqref{eq:first_weak_limit_in_Thm_5}. 

It remains to establish the limit in~\eqref{eq:weak_limit_g_and_Z} of Lemma~\ref{prop:g_t,Z(t)-limit}, which will imply Theorem~\ref{propositionConvergenceOfPairsdepsdgepsg}.


\begin{proof}[Proof of Lemma~\ref{prop:g_t,Z(t)-limit}]
 We will  establish~\eqref{eq:weak_limit_g_and_Z} using  Theorem~\ref{lemmaYaglomLimit} and Proposition~\ref{lemmaCBIExcursionStraddlingt} above.
The key insight here is as follows: given
$t-g_t(Z_n)=l$, 
the law of the excursion straddling $t$ equals 
the law of the first excursion $E_n^l$ of $Z_n$ exceeding the length $l\in(0,t)$. 
This well-known fact can be proved directly and simply for the BGWI process $Z_1$ and its scaled continuous-time version  $Z_n\stackrel{d}{=}Z_1(\floor{n\cdot})/b_n$. 
We note that this fact has been  rigorously established for Hunt processes with transition densities (such as $Z$) in Theorem~7.35 of \cite{MR525052}.

Let $\nu_n^{l_n}$ denote the law of $E_n^{l_n}(l_n)$ and 
note that 
it coincides with the law of $Z_n(l_n)$, conditioned on $d_{1/n}(Z_n)>l_n$.
Theorem~\ref{lemmaYaglomLimit} thus implies that for any sequence $l_n\to l$, $\nu_n^{l_n}$ converges weakly to the Linnik law $\nu^l$ 
with Laplace transform given by $\lambda\mapsto 1/(1+\alpha c l \lambda^\alpha)$. 
By Lemma~\ref{propHittingTimeConvergence} we have $t-g_t(Z_n)\stackrel{d}{\to} t-g_t(Z)$ and,
by the Skorokhod representation, we may assume that
$l_n=t-g_t(Z_n)\to t-g_t(Z)=l$ almost surely. 
For arbitrary continuous and bounded functions $h_1,h_2:[0,\infty)\to\re$, we have: 
$\nu_n^{t-g_t(Z_n)}(h_2)\to\nu^{t-g_t(Z)}(h_2)$ almost surely and 
the Dominated Convergence Theorem implies the limit
\begin{align}
\label{eq:key_lomiot_g_t}
   \mathbb{E}(h_1(g_t(Z_n))\nu_n^{t-g_t(Z_n)}(h_2))
    \to \mathbb{E}(h_1(g_t(Z))\nu^{t-g_t(Z)}(h_2))\qquad\text{as $n\to\infty$,}
\end{align}
where $\nu^l(h_2)$ 
denotes the integral of $h_2$ with respect to the
Linnik law $\nu^l$ and 
$\nu_n^{l_n}(h_2)$  is 
the integral of $h_2$ with respect to  the 
law $\nu_n^{l_n}$ of the first excursion $E_n^{l_n}$ of $Z_n$
exceeding $l_n$, evaluated at $l_n$ (i.e. $\nu_n^{l_n}$ is the law of  the variable $E_n^{l_n}(l_n)$).
Moreover, since
$E_n^{l_n}(l_n)$
has the same law as $Z_n(l_n)$ conditioned on $d_{1/n}(Z_n)>l_n$,
the equality
$ \mathbb{E}(h_1(g_t(Z_n))h_2(Z_n(t)))
    =\mathbb{E}(h_1(g_t(Z_n))\nu_n^{t-g_t(Z_n)}(h_2))$
    holds.
By Proposition~\ref{lemmaCBIExcursionStraddlingt} above, the Linnik law $\nu^l$ equals that of $Z(t)$, given $g_t(Z)=t-l$. 
Thus  
$\mathbb{E}(h_1(g_t(Z))\nu^{t-g_t(Z)}(h_2))=\mathbb{E}(h_1(g_t(Z)) h_2(Z(t)))$, and
by~\eqref{eq:key_lomiot_g_t} we get
\[
\mathbb{E}(h_1(g_t(Z_n))h_2(Z_n(t)))\to
	\mathbb{E}(h_1(g_t(Z)) h_2(Z(t)))\quad\text{as $n\to\infty$.}
\]
Since the functions $h_1$ and $h_2$ were arbitrary,
the weak limit
    in~\eqref{eq:weak_limit_g_and_Z} holds, concluding the proof. 
    \end{proof}

\section*{Acknowledgements}
AM is supported by EPSRC grants EP/V009478/1 and EP/W006227/1. GUB is supported by EPSRC grant EP/V009478/1 and UNAM-DGAPA-PAPIIT grant IN110625. BP is supported by the Warwick Statistics PhD award. AM and GUB would like to thank the Isaac Newton Institute for Mathematical Sciences, Cambridge, for support  during the INI satellite programme \textit{Heavy tails in machine learning}, hosted by The Alan Turing Institute, London, and the INI programme \textit{Stochastic systems for anomalous diffusion} hosted at INI in Cambridge, where work on this paper was undertaken. These programmes were supported by EPSRC grant EP/R014604/1.

\appendix
\section{On the tail behaviour of \texorpdfstring{$\mu$}{mu} and \texorpdfstring{$\nu$}{nu} and regular variation of a certain scaling sequence}
\label{appendix_aroundRegularVariation}
In this section, we will prove the  asserted equivalence between \ref{assumption_SL} and the following one when $\alpha\in (0,1)$. 
We will also see why the sequence $a_n$ normalizing the random walk local time $L_1(X_1)$, is regularly varying. 

\subsection{Regular variation of the tails of \texorpdfstring{$\mu$}{mu} and \texorpdfstring{$\nu$}{nu} in Assumptions~\ref{assumption_SL}}
\label{subsec:equivalence_SL_SL'}
Let $\mu$ and $\nu$ be the offspring and immigration distributions associated to $f$ and $g$. 
Consider the following assumption on $\mu$ and $\nu$. 
\begin{namedassumption}{\textbf{(SL')}}\label{assumption_SL'}
\begin{enumerate}
\item $\mu$ is 
critical (i.e. has mean one); 
\item its tail $\overline \mu$  is regularly varying with index $-(1+\alpha)$; 
\item the following tail balance condition holds:
\[
\frac{k\overline \nu(k)}{\overline \mu(k)}\text{ converges to a limit in $(0,\infty)$ as $k\to\infty$.}
\]
\label{item:ass_SL'(3)}
\end{enumerate}
\end{namedassumption}

\begin{proposition}
\label{porp:A1}
For $\alpha\in (0,1)$,
Assumptions~\ref{assumption_SL} and~\ref{assumption_SL'} are equivalent. In this case $\mu$ and $\nu$ are in the stable domains of attraction with index $1+\alpha$ and $\alpha$, respectively.
\end{proposition}
\begin{proof}
We rely on Tauberian theory, as found in \cite[XIII.5]{MR0270403} or \cite{MR898871}. 
This is simplified by a shift in  perspective from the offspring and immigration generating functions $f$ and $g$ to the Laplace transforms $\psi$ and $\phi$ of $\mu$ and $\nu$. 
They are related by $\psi(\lambda)=\imf{f}{e^{-\lambda}}$ and $\phi(\lambda)=g(e^{-\lambda})$. 
Note that, upon setting $s=e^{-\lambda}$, 
we have $1-s\sim \lambda$ as $\lambda\to 0$.  
In the case of $\mu$, it is convenient to define the tail and  iterated tail of $\mu$ for any $x\in(0,\infty)$ as follows: 
\[
    \overline \mu(x)
    =\imf{\mu}{(x,\infty)}
    =\int_{(x,\infty)}\, \imf{\mu}{dy} 
    \quad\text{and}\quad
    \overline{\overline{\mu}}(x)=\int_x^\infty \overline\mu(y)\, dy. 
\]When $\alpha\in (0,1)$, 
\cite[Lemma, \S XIII.5, p. 446]{MR0270403} 
implies the equivalence, as $x\to\infty$, of the following
\[
    \overline{\overline{\mu}}(x)\sim \frac{\mathcal{l}(1/x)}{x^\alpha }
    \quad\text{and}\quad
    \overline\mu(x)\sim \frac{\mathcal{l}(1/x)}{\alpha x^{1+\alpha} },
\]
implying $\alpha \overline{\mu}(x)\sim \overline{\overline{\mu}}(x)/x$.
The tail balance condition 
in~\eqref{item:ass_SL'(3)} of Assumption~\ref{assumption_SL'}
is then equivalent to the existence of the limit $\overline\nu(x)/\overline{\overline{\mu}}(x)$ as $x\to \infty$. 

On the other hand, a direct computation gives
\[
    \int_0^\infty  \overline{\overline{\mu}}(x) e^{-\lambda x} \, dx
    =\frac{1}{\lambda^2}\int_0^\infty [e^{-\lambda x}-1+\lambda x]\, \mu(dx)
    =\frac{1}{\lambda^2}[\psi(\lambda)-1+\lambda ], 
\]so that the Tauberian theorem (for densities as in \cite[Thm~XIII.5.4]{MR0270403}), 
gives us the equivalence (for $\alpha\in (0,1)$) of
\[
    \lambda -1-\psi(\lambda)
    \sim \lambda^{1+\alpha}\mathcal{l}(\lambda)
    \quad\text{and}\quad
   \overline{\overline{\mu}}(x)\sim \frac{\mathcal{l}(1/x)}{\Gamma(1-\alpha) x^{\alpha}}. 
\]
But then, the tail balance condition of \ref{assumption_SL'} and the asymptotic equivalence condition of \ref{assumption_SL} are seen to be equivalent. 

Of course, when $\alpha\in (0,1)$, either \ref{assumption_SL} or \ref{assumption_SL'}
imply that both $\mu$ and $\nu$ are in the domain of attraction of stable distribution. 
Indeed, 
define $b_n$ implicitly by $b_n^\alpha=n \mathcal{l}(1/b_n)$ 
and compute Laplace transforms: 
\begin{equation}
\begin{split}
    g(e^{-\lambda/b_n})^n
    &=\left[ 1-(1-g(e^{-\lambda/b_n})) \right]^n
    =\left[ 1-d(1-e^{-\lambda/b_n})^\alpha \mathcal{l}(1-e^{-\lambda/b_n}) \right]^n
    \\&\sim \left[ 1-d(\lambda/b_n)^\alpha \mathcal{l}(1/b_n) \right]^n
    \to e^{-d\lambda^\alpha}\quad\text{as $n\to\infty$.} 
    \label{eq:g_conv}
\end{split}
\end{equation} 
A similar argument
using the moment generating function $f$ of $\mu$, 
tells us that for the function $\tilde f(s)=f(s)/s$ we have
\begin{equation}
\label{eq:f_conv}
    \tilde f(e^{-\lambda/b_n})^{n b_n}\to e^{c \lambda^{1+\alpha}}\quad\text{as $n\to\infty$.}
    \qedhere
\end{equation}
\end{proof}

\subsection{Behaviour of \texorpdfstring{$\mu$}{mu} and \texorpdfstring{$\nu$}{nu} for \texorpdfstring{$\alpha=1$}{alpha equals 1} in Assumptions~\ref{assumption_SL}}
\label{app:alpha=1}
We identify the cases when the immigration law $\nu$ has a finite mean and the offspring law $\mu$ has a finite variance. Recall
\begin{equation*}
\sum_{n\in\na}n^2\mu(n)=\infty\iff f''(1)=\infty \quad\&\quad 
\sum_{n\in\na}n\nu(n)=\infty\iff g'(1)=\infty.
\end{equation*}
\begin{lemma}\label{lemma_infinite_activity} Under Assumption \ref{assumption_SL}, for any $\alpha\in(0,1]$ we have 
\begin{equation}
  \label{eq:first_equiv}  
f''(1) = \infty \iff 
g'(1) = \infty \iff
\text{$\alpha < 1$ or $\lim_{u\downarrow0}{\mathcal{l}}(u) = \infty$.}
\end{equation}
Thus, $\mu$ and $\nu$ have finite variance and mean, respectively, if and only if $\alpha=1$ and   $$\exists \lim_{u\downarrow0}{\mathcal{l}}(u)\in(0,\infty)\quad\text{(or equivalently $\exists\lim_{u\downarrow0}{\mathcal{k}}(u) \in(0, \infty)$).}$$
\end{lemma}


\begin{remark}
As observed in~\cite[p.~142,~Remark]{MR0228077},  it is possible to have $\alpha=1$  with offspring distribution of infinite variance, e.g.~$f(s) \coloneqq s + c (1-s)^2 \paren{1 -  c'\log\paren{1-s}}$
for $c'\in(0,2/3)$  and $0<c'(2-c)<1$. 
Similarly,   for any $d\in(0,1)$, $g(s) \coloneqq 1-d(1-s)(1 - c'
\log(1-s))$ is a  generating function with $g'(1)=\infty$. The corresponding pair of the offspring and immigration laws $\mu$ and $\nu$ with  infinite variance and mean, respectively,   satisfy Assumptions~\ref{assumption_SL} with $\alpha=1$.
\end{remark}

\begin{proof}[Proof of Lemma~\ref{lemma_infinite_activity}]
Assume $\alpha\in(0,1)$. Then, by Lagrange's theorem, for any $s\in(0, 1)$ there exists $\theta_s\in (s, 1)$  such that
$g'(\theta_s) = \frac{1 - g(s)}{1 - s} = d (1 - s)^{\alpha - 1} {\mathcal{k}}(1- s)$.
Hence, since $\mathcal{k}$ is slowly varying at zero, letting $s\uparrow 1$ we see that $g'(1) =\infty$. Similarly, we can show that $f''(1) = \infty$. 

Now assume $\alpha= 1$. By a straightforward computation, we obtain
\begin{equation}\label{lemma_infinite_activity:f_2prime}
f''(1-x) = {\mathcal{l}}(x) \paren{2 c + 4 c \frac{x {\mathcal{l}}'(x)}{{\mathcal{l}}(x)}+ c \frac{x^2
{\mathcal{l}}''(x)}{{\mathcal{l}}(x)}} \ \text{and} \ g'(x) = {\mathcal{k}}(x)\paren{d+ d\frac{x {\mathcal{k}}'(x)}{{\mathcal{k}}(x)}}.
\end{equation} 

The following statement holds by~\cite[Thm~2]{MR0094863}:
for any $\gamma>0$ and any slowly varying function $\mathcal{h}$ (at zero), such that the regularly varying function $H(x)\coloneqq x^\gamma \mathcal{h}(x)$ has a derivative which is monotone near $0$, we have
\begin{equation}
\label{eq:slow_vay_H}
    \lim_{x\to0} \frac{x H'(x)}{H(x)} = \gamma.
\end{equation}

To understand the behaviour of $f''$ near $1$, consider  $F(x)  \coloneqq  f(1 - x) - (1 - x)= x^2 \mathcal{l}(x)$ and note that $F'(x) = 1 - f'(1-x)$. Since
$f'$ is increasing, $F'$ is monotone in the neighbourhood of $0$. By~\eqref{eq:slow_vay_H}  
applied to $F$ (with $\gamma =2$ and $\mathcal{h}=\mathcal{l}$), we obtain
\begin{equation}
\label{eq_lamperti_slowly_varying} 
2 = \lim_{x\to0+} \frac{x F'(x)}{F(x)} = \lim_{x\downarrow0} \frac{x
    \paren{x^{2} \mathcal{l}(x)}'}{x^{2} \mathcal{l}(x)} = 2 +
    \lim_{x\downarrow0} \frac{x \mathcal{l}'(x)}{\mathcal{l}(x)},\qquad\text{implying $\lim_{x\downarrow0}\frac{x \mathcal{l}'(x)}{\mathcal{l}(x)}=0$.} 
    \end{equation}
Similarly, applying~\eqref{eq:slow_vay_H}
to $G(x) \coloneqq 1 -
g(1-x)$, we get $\lim_{x\to0+} \paren{x \mathcal{k}'(x)}/{\mathcal{k}(x)} = 0$.

It is clear that  $F''(x) = f''(1-x)$ is monotone and, by~\eqref{eq_lamperti_slowly_varying},  $F'$ is proportional at zero to $F(x)/x$ and thus regularly varying with index $\alpha=1$.
 By~\eqref{eq:slow_vay_H} applied to $F'$ (with $\gamma=1$ and $\mathcal{h}(x)=2c\mathcal{l}(x)+cx\mathcal{l}'(x))$ we get
\[ 2 = \lim_{x\downarrow0} \frac{x F''(x)}{F'(x)} = 
    \lim_{x\downarrow0} \frac{2 c x {\mathcal{l}}(x) + 4 c x^2 {\mathcal{l}}'(x) + c x^3 {\mathcal{l}}''(x)}{2 c x {\mathcal{l}}(x) + c x^2 {\mathcal{l}}'(x) } = 1 + 
    \lim_{x\downarrow0} \frac{x^2 {\mathcal{l}}''(x)}{2 {\mathcal{l}}(x)}, \]
    where the last equality holds using the last limit in~\eqref{eq_lamperti_slowly_varying}.
We thus get $\lim_{x\downarrow0} \frac{x^2 {\mathcal{l}}''(x)}{{\mathcal{l}}(x)} = 2$. 

Recall that $c$ and $d$ are positive and that $\mathcal{l}(x)\sim \mathcal{k}(x)$ as $x\to 0$. Then~\eqref{lemma_infinite_activity:f_2prime} implies: 
$ f''(1) = \infty \Leftrightarrow \lim_{x\downarrow0}{\mathcal{l}}(x)=\infty \Leftrightarrow
    \lim_{x\downarrow0}{\mathcal{k}}(x)=\infty \Leftrightarrow g'(1) = \infty$,
    proving~\eqref{eq:first_equiv}.

It follows from~\eqref{eq:first_equiv} that $\nu$ and $\mu$ have finite mean and variance, respectively, if and only if $\alpha=1$ and $\limsup_{u\downarrow0}\mathcal{l}(u)<\infty$. If the latter holds, by~\eqref{lemma_infinite_activity:f_2prime} we have  $\lim_{u\downarrow0}\mathcal{l}(u)=\frac{1}{2c}f''(1)<\infty$, concluding the proof of the lemma.
\end{proof}

\subsection{Regular variation of the sequence \texorpdfstring{$(a_n)$}{(an)} in Theorem~\ref{septupleLimitTheorem}}
\label{subce:a_n_reg_var}
Since the finite-dimensional distributions of $L_1(X_1)(\floor{nt})/a_n$  converge to the non-trivial limit $L$, 
\cite[Thm~2]{MR0138128} 
implies the regular variation of $(a_n)$
with index equal to the self-similarity index of $L(X)$.
We now give two proofs that the latter equals $1+1/\alpha$.

Since $X$ is self-similar of index $\beta=1+\alpha$, the density $f_t$ of $X_t$ satisfies:
\[
    f_t(x)=f_1(xt^{-1/\beta})t^{-1/\beta}. 
\]With this, one sees that the resolvent densities
\[
    u_\lambda(x,y)=\int_0^\infty e^{-\lambda t} f_t(y-x)\, dt
\]are bicontinuous and one can use the occupation times formula for the local time field of $X$ to get that the Laplace exponent of inverse local time is $1/u_\lambda(0,0)$. 
Since
\[
    u_\lambda(0,0)
    =\int_0^\infty e^{-\lambda t} f_t(0)\, dt
    =\int_0^\infty e^{-\lambda t} f_1(0)t^{-1/\beta}\, dt
    =c \lambda^{-(1-1/\beta) },
\]where $c=\Gamma(1-1/\beta)f_1(0)$, we see that the inverse local time at zero of $X$ is a stable subordinator of index $1-1/\beta$. Hence, the local time $L(X)$ is self-similar of index $1/(1-1/\beta)=1+1/\alpha$.

\section{A sequence of BGWIs with a self-similar CBI limit and local times that do not converge to the local time of the CBI}
\label{app:example_non_conv_local_time}

We construct a sequence of critical BGWIs which, although convergent to a self-similar CBI, is such that its sequence of counting local times does not converge to the local time of the  CBI. The construction is based on a BGWI in the domain of attraction (i.e. satisfying Assumption~\ref{assumption_SL}), which we perturb by a single immigrant with large positive probability at each time step, see~\eqref{eq:one_more_immigrant} below. The new immigration law  retains the tails of the original  immigration distribution, while the offspring distribution remains the same along the sequence. Displacing the pre-limit process away from $0$ with positive probability (which tends to one along the sequence of the BGWIs) makes the corresponding scaled BGWIs eventually almost surely not visit zero during  a given compact time interval. 

More specifically, it is easy to see from Taylor's theorem that 
\begin{equation}
\label{eq:gen_functions}
f(s) := s + c (1 - s)^{1 +\alpha}, \quad  g(s) := 1 - d (1 - s)^{\alpha} \quad \&\quad 
g_m(s) := 1 - p_m + p_m s g(s),
\end{equation}
for $s\in(0,1)$,
are generating functions if $\alpha\in(0,1)$ and
\begin{equation*}
d, c (1 + \alpha)\in(0,1)  \ \ \text{and $0<p_m<1$ for all $m\in\na$.} 
\end{equation*}
The immigration generating function  $g$ describes the law of the increment the random walk $Y_1$ in the discrete Lamperti transform in~\eqref{eq_discreteLampertiTransformation_1}. 
The increment of the  immigration random walk $Y^{(m)}$, corresponding to
the  generating function $g_m$,
satisfies:
\begin{equation}
\label{eq:one_more_immigrant}
\text{$\proba{Y^{(m)}(1)=1+Y_1(1)}=p_m$}\qquad\&\qquad \text{$\proba{Y^{(m)}(1)=0}=1-p_m$.}
\end{equation}

Let
$Z^{(m)}$
be a discrete-time BGWI process with offspring  $f$ and immigration $g_m$.
Consider its scaled continuous-time extension $Z^{(m)}(\floor{m\cdot})/b_m$, where 
the sequence $(b_m)$ satisfies~\eqref{equationDefbn} as $m\to\infty$.
If $p_m\to1$ as $m\to\infty$,
the sequence $(Z^{(m)}(\floor{m\cdot})/b_m)_{m\geq1}$ converges weakly to the same limit as the sequence $(Z_1(\floor{m\cdot})/b_m))_{m\geq1}$, where 
$Z_1$ is the BGWI in~\eqref{eq:cont_time_bgwi} corresponding to $f$ and $g$ in~\eqref{eq:gen_functions}.
Indeed, by~\cite[Thm~2.1]{MR2225068}, it suffices to check that $m/b_m\to0$  and that the functions $F_{b_m}(\lambda) \coloneqq m \paren{ 1- g_m(1 - \lambda / b_m)}$ converge,
\[
\begin{split}
F_{b_m}(\lambda) & = m p_m (1 - (1 - \lambda / b_m) g(1 - \lambda / b_m)) 
= m p_m (1 - (1 - \lambda / b_m) (1 - d (\lambda / b_m)^\alpha)) \\
& = m p_m (\lambda / b_m + d (\lambda / b_m)^\alpha - d (\lambda / b_m)^{1+\alpha}) 
\to d  \lambda^\alpha, \ \ \text{as $m\to\infty$,} 
\end{split}
\]
 for all $\lambda>0$,
where in both limits we used $b_m^\alpha \sim m$ as $m\to\infty$ from~\eqref{equationDefbn} and $\alpha < 1$. Thus the weak limit of the sequence $(Z^{(m)}(\floor{m\cdot})/b_m)_{m\geq1}$ of BGWIs is a self-similar CBI with $\delta = \frac{d }{\alpha c}$ and marginals given by~\eqref{equationLaplaceTransformOfCBI}. As noted in the introduction, by~\cite[\S5.2.1]{MR3263091}, if $\delta \in (0, 1)$, the limiting process is point recurrent at $0$ with non-degenerate Markov local time $L$. 

\begin{lemma}
\label{lem:final}
Set $p_m\coloneqq 1-m^{-3}$ for $m\geq1$. Let  $L^{(m)}(t)  \coloneqq  \left|\set{k \in\na: Z^{(m)}(k) = 0}\cap [0, t]\right|$ be the counting local time at zero of the BGWI process $Z^{(m)}$ constructed above. Then 
\begin{equation}
\label{eq:final}
\proba{\cup_{k=1}^\infty\cap_{m=k}^\infty\{L^{(m)}(m)=1\}}=1.
\end{equation}
\end{lemma}

Under the assumption of Lemma~\ref{lem:final}, the sequence of counting local times 
$L^{(m)}(\floor{m\cdot})$ at time $1$, equal to $L^{(m)}(m)$, is  eventually equal to one for all large values of  $m$ and hence almost surely converges to $1$  as $m\to\infty$.  Thus, when scaled by any sequence that tends to infinity, $L^{(m)}(m)$ cannot converge weakly to the local time  $L$ of the limiting CBI at time $1$, which is a non-trivial random variable.

\begin{proof}[Proof of Lemma~\ref{lem:final}]  
It is clear from the Lamperti transform in~\eqref{eq_discreteLampertiTransformation_1}
for BGWI $Z^{(m)}$ that the inclusion 
$\{L^{(m)}(m)=1\}^c=\{L^{(m)}(m)>1\}=\{\min_{1 \leq k < m} Z^{(m)}(k) = 0\}\subset A_m^c$ holds, where $$A_m\coloneqq \{\forall k\in\{1,\ldots,m\}:Y^{(m)}(k)-Y^{(m)}(k-1)>0\}.$$
By~\eqref{eq:one_more_immigrant} and our assumption on $p_m$ we have  
$$\proba{A_m}=p_m^m=(1-m^{-3})^m=\left((1-m^{-3})^{m^3}\right)^{m^{-2}}\sim e^{-m^{-2}},\quad\text{as $m\to\infty$,}$$
implying 
$\proba{A_m^c}\sim 1-e^{-m^{-2}}\sim m^{-2}$. 
In particular, $\sum_{m=1}^\infty\proba{A_m^c}<\infty$.
The Borel-Cantelli lemma implies
$\proba{\cap_{k=1}^\infty\cup_{m=k}^\infty\{L^{(m)}(m)=1\}^c}\leq \proba{\cap_{k=1}^\infty\cup_{m=k}^\infty A_m^c}=0$ and~\eqref{eq:final} follows.
\end{proof}

\bibliography{GenBib}
\bibliographystyle{amsalpha}
\end{document}